\documentclass[letterpaper,11pt]{amsart}
\usepackage{indentfirst} 
\usepackage{amssymb}
\usepackage{mathrsfs}
\usepackage{amsmath}
\usepackage{amsthm}
\usepackage{thmtools}
\usepackage{enumitem}
\usepackage{bbm}            
\usepackage[colorlinks = true, linkcolor = black, citecolor = blue ]{hyperref}
\usepackage[usenames,dvipsnames]{xcolor} 
\usepackage{tikz}
\usepackage[left=3.2cm, right=3.2cm, bottom=3.4cm]{geometry}

\theoremstyle{plain}
\newtheorem{theorem}{Theorem}[section]

\theoremstyle{plain}
\newtheorem{proposition}{Proposition}[subsection]
\newtheorem{lemma}[proposition]{Lemma}
\newtheorem{corollary}[proposition]{Corollary}

\theoremstyle{definition}
\newtheorem{definition}{Definition}[subsection]

\theoremstyle{remark}
\newtheorem{remark}{Remark}[section]

\newtheorem{example}{Example}

\newcommand{\vertiii}[1]{{\left\vert\kern-0.25ex\left\vert\kern-0.25ex\left\vert #1 
    \right\vert\kern-0.25ex\right\vert\kern-0.25ex\right\vert}}

\def\e{\epsilon}
\def\R{{\mathbb R}}
\def\N{{\mathbb N}}
\def\Z{{\mathbb Z}}
\def\M{{\mathcal M}}

\def\F{{\mathcal F}}

\def\H{{\mathcal H}}

\def\A{{\mathcal A}}

\begin{document}

\title{Pressure at infinity on countable Markov shifts}
\date{\today}

 \author[A.Velozo]{Anibal Velozo}  \address{Facultad de Matem\'aticas, Pontificia Universidad Cat\'olica de Chile, Avenida Vicu\~na Mackenna 4860, Santiago, Chile.}
\email{\href{apvelozo@uc.cl}{apvelozo@uc.cl}}
\urladdr{\href{https://sites.google.com/view/apvelozo/inicio}{https://sites.google.com/view/apvelozo/inicio}}

\begin{abstract}
In this article, we study the pressure at infinity of potentials defined over countable Markov shifts. We establish an upper semi-continuity result concerning the limiting behaviour of the pressure of invariant probability measures, where the escape of mass is controlled by the pressure at infinity. As a consequence, we establish criteria for the existence of equilibrium states and maximizing measures for uniformly continuous potentials. Additionally, we study the pressure at infinity of suspension flows defined over countable Markov shifts and prove an upper semi-continuity result for the pressure map. 
\end{abstract}

\maketitle

\setcounter{tocdepth}{1}
\tableofcontents

\section{Introduction}

Symbolic dynamics has a wide range of applications in the study of smooth dynamical systems.   For instance, a well known result of Bowen \cite{bo0} states that any uniformly hyperbolic diffeomorphism on a compact  manifold admits a Markov partition. The Markov partition provides a finite-to-one semiconjugacy between a subshift of finite type and most of the original system. Several applications to the study of the ergodic theory and thermodynamic formalism of the diffeomorphism follow from the study of symbolic dynamics (see \cite{bo2}). Similarly, a uniformly hyperbolic flow on a compact manifold can be coded as a suspension flow over a subshift of finite type (see \cite{bo3,ra}). In particular, this applies to the geodesic flow on a closed negatively curved manifold. For striking applications of the coding to the counting of periodic orbits see \cite{pp, la}. 

More recently, Sarig \cite{sa4} constructed Markov partitions for surface diffeomorphisms with positive topological entropy. Specifically, for any given $\e>0$ he constructed a countable Markov shift wherein every ergodic measure with entropy greater than $\e$ can be lifted to the symbolic space. Remarkable dynamical results have been obtained making use of this coding. For instance, Buzzi, Crovisier and Sarig \cite{bcs} proved that a positive entropy $C^\infty$ diffeomorphism of a closed surface admits at most finitely many ergodic measures of maximal entropy. It is worth pointing out that the use of countable Markov shifts to code the dynamics is natural in the non-uniformly hyperbolic setting, as subsets with sufficient hyperbolicity are non-compact in phase space. In recent years, several other non-uniformly hyperbolic dynamical systems have been coded with countable Markov shifts, or suspension flows over a countable Markov shift in the case of flows. For instance, see the work of Ledrappier-Lima-Sarig \cite{lls},  Lima-Matheus \cite{lm}, Ovadia \cite{o}, Lima-Sarig \cite{ls}, to mention a few.

In a slightly different direction, inducing schemes have been widely used in the study of the thermodynamic formalism of the geometric potential, for instance see \cite{y, ps, bt}. In this case, the lifting process allows to transfer information to the full shift on a countable alphabet, which is an example of particular interest in this work.

\subsection{Statements of the main results} We study various aspects of the thermodynamic formalism of countable Markov shifts (referred to as CMS). This subject has been extensively investigated by Gurevich \cite{gu}, Gurevich-Savchenko \cite{gs}, Walters \cite{wa}, Mauldin-Urba\'nski \cite{mu1}, Sarig \cite{sa1, sa2, sa3}, Sarig-Buzzi \cite{bs}, and several other mathematicians. Countable Markov shifts are non-compact generalizations of subshifts of finite type.  In contrast to their compact counterparts, they can have infinite topological entropy and may not be locally compact. Precise definitions of the terminology used here are provided in Section \ref{preli}.

Let $(\Sigma,\sigma)$ be a (one-sided) countable Markov shift. Equivalently, $\Sigma$ is the space of infinite (one-sided) walks on a directed graph with countably many vertices and $\sigma:\Sigma\to\Sigma$  is the shift map. We refer to a function $\phi:\Sigma\to\R$ as a potential. The measure-theoretic pressure of $\phi$ is defined by 
\begin{align*} P(\phi)=\sup_{\mu\in \M(\sigma)}\left\{h_\mu(\sigma)+\int \phi d\mu:\int\phi d\mu>-\infty\right\},\end{align*}
where $\M(\sigma)$ is the space of invariant probability measures of $(\Sigma,\sigma)$ and $h_\mu(\sigma)$ denotes the measure theoretic entropy of $\mu$.  Let $\M_{\phi}(\sigma)=\{\mu\in\M(\sigma):\int \phi d\mu>-\infty\}$. We say that $\mu\in \M_\phi(\sigma)$ is an equilibrium state of $\phi$ if $P(\phi)=h_{\mu}(\sigma)+\int \phi d\mu$. In physical terms, an equilibrium state maximizes the free energy of the system (see \cite{ruelle}). We are interested in properties of the pressure, equilibrium states and the pressure map of $\phi$: $$\mu\in \M_{\phi}(\sigma)\mapsto h_{\mu}(\sigma)+\int \phi d\mu.$$

An important difference between the thermodynamic formalism of subshifts of finite type and countable Markov shifts is that, in the latter, regular potentials may not admit equilibrium states. In its simplest form, a countable Markov shift can have finite entropy but no measure of maximal entropy. Sarig \cite{sa1} classified potentials with summable variations into three categories based on the convergence of certain series: transient, null recurrent, and positive recurrent. Positive recurrent potentials are those that admit an equilibrium state, and this is unique (see \cite{bs}). In general, by the definition of the measure-theoretic pressure, we are guaranteed the existence of sequences of measures  $(\mu_n)_n$ in $\M_\phi(\sigma)$ such that $$P(\phi)=\lim_{n\to\infty}\bigg( h_{\mu_n}(\sigma)+\int \phi d\mu_n\bigg).$$
For subshifts of finite type and continuous potentials $\phi$, it is well known that every limit measure of the sequence $(\mu_n)_n$  is an equilibrium state. This follows directly from the upper semicontinuity of the entropy map and the weak$^*$ compactness of $\M(\sigma)$. We are interested in the behavior of the sequence $(\mu_n)_n$ for a general countable Markov shift and a potential $\phi$ with mild regularity assumptions. More precisely, consider the following class of potentials:
$$\label{H}\H=\{\phi\in C_{uc}(\Sigma): \sup(\phi), \text{var}_2(\phi),P(\phi)<\infty\},$$
where $C_{uc}(\Sigma)$ is the space of uniformly continuous potentials on $\Sigma$, equipped with the metric $d$ defined in Section \ref{preli}, and $\text{var}_n(\phi)$ denotes the $n$-th variation of $\phi$. We will study the problem of existence and non-existence of equilibrium states for potentials in $\H$.

\subsubsection{Compactness in the space of invariant  measures}
For a transitive countable Markov shift, the space of invariant probability measures is non-compact. In the case of a countable Markov shift with finite entropy, this can be fixed by considering the space of invariant sub-probability measures endowed with the topology of convergence on cylinders, which is a compact metric space (see \cite[Theorem 1.2]{iv} and subsection \ref{introtop}). We say that a sequence $(\mu_n)_n$ converges on cylinders to $\mu$ if $\lim_{n\to\infty}\mu_n(C)=\mu(C)$, for every cylinder $C\subseteq \Sigma$. The topology of convergence on cylinders, or cylinder topology, is the topology which induces this notion of convergence. In \cite{iv}, the authors studied the topology of convergence on cylinders and identified the class of countable Markov shifts for which the space of sub-probability measures is compact under this topology: countable Markov shifts with the $\F-$property. It turns out that for countable Markov shifts without the $\F-$property, $\M(\sigma)$ cannot be compactified using the cylinder topology, as there are sequences in $\M(\sigma)$ that converge to functionals that are finitely additive but not countably additive (this is the case of the full shift; see \cite[Proposition 4.19]{iv}). In order to address the lack of a natural compactification for $\M( \sigma)$, we prove the following compactness result:

\begin{theorem}\label{thm:bsctm} Let $(\Sigma,\sigma)$ be a transitive countable Markov shift and $\phi\in\H$. Let $(\mu_n)_n$ be a sequence in $\M_\phi(\sigma)$ such that $\liminf_{n\to\infty}\int \phi d\mu_n$ is finite. Then, $(\mu_n)_n$ has a subsequence which converges on cylinders to a measure $\mu\in\M_{\le 1}(\sigma)$, where $\int \phi d\mu>-\infty$.
\end{theorem} 

\subsubsection{Pressure at infinity, escape of mass and equilibrium states}\label{sub1}
We study the pressure at infinity of a potential,  generalizing the concept of entropy at infinity of $(\Sigma,\sigma)$, as studied in \cite{r, b, itv}, to non-zero potentials. We define the measure theoretic pressure at infinity as follows:
\begin{align*}P_\infty(\phi)=\sup_{(\mu_n)_n\to 0}\limsup_{n\to\infty}\bigg(h_{\mu_n}(\sigma)+\int \phi d\mu_n\bigg),\end{align*}
where the supremum runs over all sequences $(\mu_n)_n$ in $\M_{\phi}(\sigma)$ which converge on cylinders to the zero measure. If there is no such sequence, we set $P_\infty(\phi)=-\infty$.  In Section \ref{secpre}, we define the topological pressure at infinity of $\phi$,  denoted  by $P^{top}_\infty(\phi)$, which is analogous to the topological pressure, or Gurevich pressure, of $\phi$. We prove a variational principle for the pressures at infinity.

\begin{theorem}\label{varpriplus} Let $(\Sigma,\sigma)$ be a transitive countable Markov shift and $\phi:\Sigma\to\R$ a potential with summable variations and finite pressure. Then $P_\infty(\phi)=P^{top}_\infty(\phi)$. 
\end{theorem}

For a potential $\phi:\Sigma\to\R$ with finite pressure, we define $$s_\infty(\phi)=\inf\{t\in [0,1]:P(t\phi)<\infty\}.$$ The following result relates the escape of mass phenomenon with the upper semi-continuity of the pressure map.

\begin{theorem} \label{thm:1}  Let $(\Sigma,\sigma)$ be a transitive countable Markov shift and $\phi\in C_{uc}(\Sigma)$ a potential such that $\emph{var}_2(\phi)$ and $P(\phi)$ are finite.  Let $(\mu_n)_n$ be a sequence  in $\M_\phi(\sigma)$ which converges on cylinders to $\lambda\mu$, where $\lambda\in [0,1]$ and $\mu\in\M_\phi(\sigma)$. Suppose that $\liminf_{n\to\infty}\int \phi  d\mu_n>-\infty$, or that $s_\infty(\phi)<1$. Then,
$$\limsup_{n\to\infty}\bigg(h_{\mu_n}(\sigma)+\int \phi d\mu_n\bigg)\le \lambda\bigg(h_\mu(\sigma)+\int \phi d\mu\bigg)+(1-\lambda)P_\infty(\phi).$$
Moreover, the inequality is sharp.
\end{theorem}

Theorem \ref{thm:1} is a generalization of  \cite[Theorem 1.1]{itv}, where the authors proved that the result holds when $(\Sigma,\sigma)$ has finite topological entropy and $\phi=0$. In our case  the finite entropy condition is replaced by a finite pressure assumption.  In the context of homogeneous dynamics and for the geodesic flow on non-compact  pinched negatively curved manifolds similar results  have been obtained in \cite{elmv, ek, ekp,rv,ve, gst}. The pressure at infinity has also been studied in the context of countable Markov shifts in \cite{rs} and for the geodesic flow on non-compact pinched negatively curved manifolds in \cite{gst}, where a similar result was obtained. 

In \cite{sa2}, Sarig introduced the class of strongly positive recurrent potentials (referred to as SPR). The thermodynamic formalism of SPR potentials is similar to that of H\"older potentials in subshifts of finite type. Spectral gap on suitable Banach spaces holds for SPR potentials and the RPF measure has exponential decay of correlations (see \cite{sa2, cs}). It was recently proved in \cite{rs} that if $\phi$ is weakly H\"older, bounded above and with finite pressure,  then $\phi$ is SPR if and only if $P_\infty(\phi)<P(\phi)$. In general, we say that $\phi\in C_{uc}(\Sigma)$  is SPR if $P(\phi)$ is finite and $P_\infty(\phi)<P(\phi)$. The following result is a consequence of Theorem \ref{thm:bsctm} and Theorem \ref{thm:1}: it provides a criterion for the existence of equilibrium states and describes what occurs when no equilibrium state exists.
 
\begin{theorem}\label{thm:2} Let $(\Sigma,\sigma)$ be a transitive CMS and $\phi\in\H$. Let $(\mu_n)_n$ be a sequence in $\M_{\phi}(\sigma)$ such that $P(\phi)=\lim_{n\to\infty} \big(h_{\mu_n}(\sigma)+\int \phi d\mu_n\big).$ Then:
\begin{enumerate}
\item\label{122} If $\phi$ is SPR and either $s_\infty(\phi)<1$, or $\liminf_{n\to\infty}\int \phi  d\mu_n>-\infty$, then  $(\mu_n)_n$ has a subsequence that converges in the weak$^*$ topology to an equilibrium state of $\phi$. In particular, $\phi$ has an equilibrium state.
\item \label{222} If $\phi$ does not have an equilibrium state, then there exists a subsequence  $(\mu_{n_k})_k$ that either converges to the zero measure, or satisfies that $\lim_{k\to\infty}\int \phi d\mu_{n_k}=-\infty$. 
\end{enumerate}
\end{theorem}

\subsection{Some applications of the main results} As a consequence of our methods, we derive applications in ergodic optimization and the thermodynamic formalism of suspension flows over countable Markov shifts. 

\subsubsection{Ergodic optimization} For $\phi\in C_{uc}(\Sigma)$ define $$\beta(\phi)=\sup_{\mu\in \M(\sigma)}\int \phi d\mu.$$ We say that $\mu\in\M(\sigma)$ is a maximizing measure for $\phi$ if $\beta(\phi)=\int \phi d\mu$.  Maximizing measures are a central object of study in ergodic optimization. Natural questions in this context include what conditions ensure the existence of a maximizing measure, and if such measures exist, what can be said about their dynamical properties and support. For a general account of ergodic optimization, we refer the reader to \cite{j} and  \cite{bo}.

For a continuous and compact dynamical system, the existence of maximizing measures for continuous potentials is a standard result, which follows from the weak$^*$ compactness of the space of invariant probability measures and the continuity of the integral map. In our setting, however, $\M(\sigma)$ is non-compact, and this same argument does not apply. In fact, it is relatively straightforward to construct potentials that do not admit any maximizing measure (see Example \ref{ex:nomea}). In general, hypotheses on the potential are necessary to guarantee the existence of a maximizing measure. This problem has been studied for the renewal shift in \cite{i}. For different classes of coercive potentials, the existence of maximizing measures was established in \cite{jmu, bf, fv}.  

In analogy to the pressure at infinity, we define
$$\beta_\infty(\phi)=\sup_{(\mu_n)_n\to0}\limsup_{n\to\infty}\int \phi d\mu_n,$$
where the supremum runs over all sequences in $\M(\sigma)$ which converge on cylinders to the zero measure. If there is no such sequence, we set $\beta_\infty(\phi)=-\infty$. The quantity $\beta_\infty(\phi)$ was introduced in \cite{rv2} in the context of the geodesic flow on non-compact pinched negatively curved manifolds. In \cite{rv2}, the authors proved that for uniformly continuous potentials that are bounded above, the condition $\beta_\infty(\phi)<\beta(\phi)$ implies the existence of a maximizing measure for $\phi$. In Section \ref{maxme}, we prove an inequality that relates the upper semi-continuity of the integral map $\mu\mapsto\int \phi d\mu$ to $\beta_\infty(\phi)$, which is analogous to Theorem \ref{thm:1} (see Theorem \ref{thm:maxm}). As an application, we obtain the a criterion for the existence of a maximizing measure.

\begin{theorem}\label{thm:mm} Let $(\Sigma,\sigma)$ be a transitive CMS and $\phi\in\H$ a potential such that $\beta_\infty(\phi)<\beta(\phi)$. Then, $\phi$ has a maximizing measure.
\end{theorem}

A finer description, analogous to Theorem \ref{thm:2}, is given in Theorem \ref{maxva}. The assumptions in Theorem \ref{thm:mm} allow for the consideration of  potentials with any behavior at infinity. For a coercive potential, we have $\beta_\infty(\phi)=-\infty$. On the other hand, in Proposition \ref{fv1}, we prove that zero temperature limits of potentials in $\H$ with summable variations and such that $\beta_\infty(\phi)<\beta(\phi),$ are maximizing measures (see also \cite[Theorem 1]{fv}).

\subsubsection{Suspension flows} In Section \ref{sec:flow}, we study aspects of the thermodynamic formalism of suspension flows defined over countable Markov shifts, a subject studied in \cite{bi, k, ij}. 

Suspension flows over a countable Markov shifts are continuous-time dynamical systems where the base dynamics is a countable Markov shift, and a roof function $\tau:\Sigma\to (0,\infty)$ specifies the return time to the base. We denote the suspension flow associated with $(\Sigma,\sigma)$ and $\tau$ by $(Y,\Theta=(\theta_t)_{t\ge 0})$. The space of flow-invariant probability (respectively, sub-probability) measures on $Y$ is denoted by $\M(\Theta)$ (respectively, $\M_{\le1}(\Theta)$). In \cite[Section 6]{iv}, the authors defined a topology on $\M_{\le1}(\Theta)$ that is analogous to the cylinder topology in $\M_{\le1}(\sigma)$, which we still refer to it as the cylinder topology (for details, see Section \ref{sec:flow}). We consider the following class of roof functions:
$$\Psi=\{\psi\in C_{uc}(\Sigma): \inf \psi>0, \text{var}_2(\psi)<\infty, P(-\psi)<\infty\}.$$
If $\tau\in\Psi$, then the topological entropy of $(Y,\Theta)$ is finite. Under this assumption we prove a result similar to Theorem \ref{thm:bsctm} (see also Theorem \ref{compact}).

\begin{theorem}\label{compactv} Let  $(\Sigma,\sigma)$ be a transitive countable Markov shift with a roof function $\tau\in \Psi$. Let $(Y,\Theta)$ be the associated  suspension flow.  Then, $\M_{\le1}(\Sigma,\sigma,\tau)$ is  compact with respect to the cylinder topology. Furthermore, $\M(\Theta)$ is dense in $\M_{\le1}(\Theta)$.
\end{theorem}

We define the pressure of $\phi:Y\to\R$, by the formula
$P^\Theta(\phi)=\sup_{\nu\in\M(\Theta)} \big(h_\nu(\Theta)+\int \phi d\nu\big),$
where $h_\nu(\Theta)$ is the entropy of $\nu$ with respect to the time one map $\theta_1$. We have added the superscript $\Theta$ to emphasize that the dynamical system under consideration is the suspension flow $(Y,\Theta)$. We define the pressure at infinity of $\phi$ as
$$P^\Theta_\infty(\phi)=\sup_{(\nu_n)_n\to 0} \limsup_{n\to\infty} \bigg(h_{\nu_n}(\Theta)+\int \phi d\nu_n\bigg),$$
where the supremum runs over all sequences $(\nu_n)_n$ in $\M(\Theta)$ which converge on cylinders to the zero measure. If  there is no such sequence, we set $P^\Theta_\infty(\phi)=-\infty$. In this context we prove a result analogous to Theorem  \ref{thm:1}. Define the following class of potentials:
$$\H_Y=\{\phi\in C_b(Y): \Delta_\phi\in C_{uc}(\Sigma), \text{var}_2\Delta_\phi<\infty\},$$
where $C_b(Y)$ is the space of continuous and bounded functions on $Y$.

 \begin{theorem}\label{thm:susp} Let $\tau\in \Psi$ and $\phi\in \H_Y$. Let $(\nu_n)_n$ be a sequence in $\M(\Theta)$ which converges on cylinders to $\lambda\nu$, where $\lambda\in [0,1]$ and $\nu\in \M(\Theta)$. Then,
$$\limsup_{n\to\infty} \bigg(h_{\nu_n}(\Theta)+\int \phi d\nu_n\bigg)\le \lambda \bigg(h_{\nu}(\Theta)+\int\phi d\nu\bigg)+(1-\lambda)P_\infty^\Theta(\phi).$$
Moreover, this inequality is sharp.
\end{theorem}

In this context, we say that $\phi:Y\to\R$ is SPR if $P^{\Theta}(\phi)$ is finite and $P_\infty^{\Theta}(\phi)<P^\Theta(\phi)$. At the end of Section \ref{sec:flow} we prove that if $\tau\in\Psi$ and $\phi\in \H_Y$ is SPR, then $\phi$ has an equilibrium state (see Theorem \ref{fine:sus}).

\subsection{Organization of the article:} 
\begin{itemize}
\item Section \ref{preli}: We review known facts about the thermodynamic formalism of countable Markov shifts, the topology of convergence on cylinders, the entropy at infinity of a CMS, and finally, a useful recoding that will be needed later in the article. 
\item Section \ref{secpre}: We study the topological pressure at infinity and establish a formula in terms of the topological pressure. We also prove Theorem \ref{varpriplus}, which provides the variational principle for pressures at infinity.
\item Section \ref{sec:spr}: We characterize the existence of sequences of invariant probability measures that converge on cylinders to the zero measure, in terms of a combinatorial property of the graph associated with a countable Markov shift.
\item Section \ref{con}: We consider a Markov potential with values in $\N$, construct a new countable Markov shift, and calculate its entropy at infinity. We prove Theorem \ref{thm:bsctm}.
\item Section \ref{infinite}: Using the countable Markov shift constructed in Section \ref{con}, we prove Theorem \ref{thm:1}, the main upper semi-continuity result for the pressure map. 
\item Section \ref{maxme}: We apply the previously obtained results to prove Theorem \ref{thm:2}. We prove that $\beta_\infty(\phi)$ is the asymptotic slope of the pressure at infinity and establish an upper semi-continuity result for the integral map which is analogous to Theorem  \ref{thm:1}. Furthermore, we prove Theorem \ref{thm:mm} and obtain a  result concerning zero temperature limits of a suitable class of potentials.
\item Section \ref{sec:flow}: We prove Theorem \ref{compactv} and Theorem \ref{thm:susp}.
\end{itemize}

\section{Preliminaries}\label{preli}
In this section, we review some facts about countable Markov shifts and their thermodynamic formalism. Good references on this topic include \cite{sarigbook} and \cite[Section 2]{mu}.

\subsection{Countable Markov shifts} Let $S$ be an alphabet with countably many symbols and $M=(M_{i,j})_{i,j\in S}$  a $S\times S$ matrix  with entries $0$ or $1$. The symbolic space associated to $S$ and $M$  is given by$$\Sigma=\{(x_1,x_2,\ldots)\in S^{\N}: M_{x_i,x_{i+1}}=1,\text{ for all }i\in \N\}.$$
We endow $S$ with the discrete topology and $S^{\N}$ with the product topology. On $\Sigma$ we consider the topology induced by the natural inclusion $\Sigma\subseteq S^{\N}$. The shift map $\sigma:\Sigma\to\Sigma$ is given by $\sigma(x_1,x_2,x_3,\ldots)=(x_2,x_3,\ldots)$. The dynamical system $(\Sigma,\sigma)$ is called a {countable Markov shift} (CMS).

\begin{remark} We will always assume that $S$ is infinite. Since $\N$ has a natural order, it is sometimes convenient to identify $S$ with $\N$. Depending on the context and convenience, we might use either $S$ or $\N$ as our alphabet.
\end{remark}

Given $x\in \Sigma$, we denote by $x_i$ to its $i$-th coordinate; equivalently, the first coordinate of $\sigma^{i-1}(x)$. A cylinder of length $m$ is a set of the form $$[a_1,\ldots,a_{m}]=\{x\in\Sigma: x_i=a_i,\forall i\in\{1,\ldots ,m\}\}.$$
An admissible word is a sequence $a_1\ldots a_n$ of symbols in $S$ such that $M_{a_i,a_{i+1}}=1$, for every $i\in\{1,\ldots,n-1\}$. A cylinder $[a_1,\ldots,a_n]$ is non-empty if and only if $a_1\ldots a_n$ is an admissible word. The collection of cylinders is a basis for the topology on $\Sigma$. 

We define a metric $d$ on $\Sigma$ by declaring $d(x,y)=0$, if $x=y$; $d(x,y)=1$, if $x_1\ne y_1$; and $d(x,y)=2^{-k},$ if $k$ is the length of the longest cylinder containing $x$ and $y$. Note that balls in the metric $d$ are cylinder sets. Given a potential $\phi:\Sigma\to \R$, we define  its $n$-variation as
$$\text{var}_n(\phi)=\sup\{|\phi(x)-\phi(y)|: \forall x,y\in \Sigma \text{ such that }d(x,y)\le 2^{-n}\}.$$
Let $S_n\phi(x)=\sum_{k=0}^{n-1}\phi(\sigma^k x),$ be the Birkhoff sums of $\phi$. 

We say that $\phi$ is bounded away from zero if   $\inf \phi>0$. The space of continuous functions on $\Sigma$ is denoted by $C(\Sigma)$.  The space of bounded and continuous functions on $\Sigma$ is denoted by $C_b(\Sigma)$. For $\phi\in C_b(\Sigma),$ we consider the $C^0-$norm $\|\phi\|_0=\max_{x\in\Sigma}|\phi(x)|$, which defines the $C^0-$topology. The space of uniformly continous functions on $\Sigma$ with respect to $d$ is denoted by $C_{uc}(\Sigma)$. Note that $\phi\in C_{uc}(\Sigma)$ if and only if $\lim_{k\to\infty} \text{var}_k(\phi)=0$. Let $C_{b, uc}(\Sigma)=C_b(\Sigma)\cap C_{uc}(\Sigma)$. We say that $\phi$ has summable variations if $\sum_{n\ge 2} \text{var}_n(\phi)<\infty$. We say that $\phi$ is  weakly H\"older if there exist $\lambda\in (0,1),$ and a constant $C>0$ such that $\text{var}_n(\phi)\le C \lambda^n$, for every $n\ge 2$. A function $\phi$ is called locally constant depending on $n$ coordinates if $\text{var}_n(\phi)=0$.  We say that $\phi$ is locally constant if $\text{var}_n(\phi)=0$, for some $n\in\N$. 

In the introduction, we defined
$$\label{H}\H=\{\phi\in C_{uc}(\Sigma): \sup(\phi), \text{var}_2(\phi),P(\phi)<\infty\}.$$
Note that $\H$ is $C^0-$closed in $C(\Sigma)$. Moreover, locally constant potentials are dense in $\H$ (uniformly continuous potentials can be $C^0-$approximated by locally constant potentials). Since locally constant potentials with finite second variation are weakly H\"older, $\H$ is the closure of the space of  weakly Hölder potentials that are bounded above. We are interested in the study of the thermodynamic formalism of potentials in $\H$.

We say that $(\Sigma, \sigma)$ is topologically transitive (or transitive) if for each pair $(a,b) \in S \times S$, there exists an admissible word starting with $a$ and ending with $b$. We say that $(\Sigma, \sigma)$ is topologically mixing (or mixing) if for each pair $(a,b) \in S \times S$ there exists a number $N(a,b)$ such that for every $n \geq N(a,b)$ there is an admissible word of length $n$ starting with $a$ and ending with $b$.

A directed graph can be associated with a CMS. If $(\Sigma, \sigma)$ is the CMS with alphabet $S$ and transition matrix $M$, consider the graph $G = (V, E)$, where the set of vertices $V$ is identified with $S$, and $(i, j) \in E$ if and only if $M_{i, j} = 1$. Points in $\Sigma$ correspond to infinite forward paths over $G$, and the shift map represents the motion along these paths. Note that $(\Sigma, \sigma)$ is transitive if and only if $G$ is strongly connected.

\subsection{Topologies on the space of invariant measures}\label{introtop} We denote by $\M_{\le1}(\sigma)$ the space of invariant sub-probability measures of $(\Sigma, \sigma)$. More precisely, $\mu \in \M_{\le1}(\sigma)$ if $\mu$ is a non-negative countably additive Borel measure such that $\mu(\Sigma) \in [0,1]$, and $\mu(A) = \mu(\sigma^{-1}A)$ for every Borel set $A \subseteq \Sigma$. The mass of $\mu \in \M_{\le1}(\sigma)$ is given by $|\mu| := \mu(\Sigma)$. We denote by $\M(\sigma)$ the space of invariant probability measures of $(\Sigma,\sigma)$, that is, elements in $\M_{\le1}(\sigma)$ with mass equal to one. The zero measure is the measure that assigns zero to every Borel set, and it is the unique element in $\M_{\le1}(\sigma)$ with zero mass.

 In this article we are  interested in the behaviour of the pressure under limits of invariant probability measures. To describe such limits we use two related topologies on $\M_{\le 1}(\sigma)$. 

We say that $(\mu_n)_n$ converges in the weak$^*$ topology to $\mu$ if  $\lim_{n\to\infty} \int fd\mu_n=\int fd\mu,$ for every $f\in C_b(\Sigma)$. If $(\mu_n)_n$ is a sequence in $\M(\sigma)$, then any weak$^*$ limit is also in $\M(\sigma)$. We say that $(\mu_n)_n $ converges on cylinders to $\mu$ if $\lim_{n\to\infty}\mu_n(C)=\mu(C),$ for every cylinder $C\subseteq \Sigma$. This notion of convergence induces the topology of convergence on cylinders, or cylinder topology, on $\M_{\le1}(\sigma)$. This topology is metrizable, for instance consider the metric $\rho:\M_{\le1}(\sigma)\times \M_{\le1}(\sigma)\to\R_{\ge 0},$ given by
\begin{align}\label{metricc} \rho(\mu_1,\mu_2)=\sum_{i\in \N}\frac{1}{2^i}|\mu_1(C_i)-\mu_2(C_i)|,\end{align}
where $(C_i)_i$ is an enumeration of the set of cylinders of $\Sigma$. 

\begin{remark}\label{rem:compatibility}
If $(\mu_n)_n$ is a sequence in $\M(\sigma)$ which converges on cylinders to $\mu$, and $\mu$ is a probability measure, then $(\mu_n)_n$ converges weak$^*$ to $\mu$ (see \cite[Lemma 3.7]{iv}). 
\end{remark}

We say that $(\Sigma,\sigma)$ has the $\F-$property if, for every $a\in S$ and $n\in \N$, there are finitely many admissible words of length $n$ that start and end with $a$. Equivalently, there are only finitely many periodic points of a given period in $[a]$. 

\begin{theorem}[{\cite[Theorem 1.2]{iv}}] \label{compact} Let $(\Sigma, \sigma)$ be a transitive CMS with the $\F$-property. Then, $\M_{\le 1}(\sigma)$ is a compact metric space with respect to the cylinder topology. Furthermore, $\M(\sigma)$ is dense in $\M_{\le 1}(\sigma)$. \end{theorem}

\begin{remark} If $(\Sigma,\sigma)$ has finite topological entropy, then it satisfies the $\F-$property. In particular, Theorem \ref{compact} applies for finite entropy CMS. On the other hand, if $(\Sigma,\sigma)$ does not have the $\F-$property, then $\M_{\le 1}(\sigma)$ is non-compact (see \cite[Proposition 4.19]{iv}).
\end{remark}

Let $C_0(\Sigma)$ be the space of test functions for the cylinder topology. This space is defined as the $C^0-$closure of the set of functions that can be written as a finite sum of characteristic functions of cylinders (for a detailed description, see \cite[Section 3.4]{iv}). For instance, the function $V=\sum_{k\in\N}\frac{1}{k}1_{[k]}$ is in $C_0(\Sigma)$. A sequence $(\mu_n)_n$ in $\M(\sigma)$ converges on cylinders to $\mu \in \M_{\le 1}(\sigma)$ if and only if $\int f d\mu_n \to \int f d\mu$ for every $f \in C_0(\Sigma)$. 

We conclude this subsection with an inequality that will be needed later in the article.

\begin{lemma}\label{tired} Let $\phi\in C_{uc}(\Sigma)$ be a non-negative function. Let $(\mu_n)_n$ be a sequence in $\M(\sigma)$ which converges on cylinders to $\mu\in\M_{\le 1
}(\sigma)$. Then 
$$\liminf_{m\to\infty}\int \phi d\mu_m\ge \int \phi d\mu$$
\end{lemma}
\begin{proof}  Since $\phi\in C_{uc}(\Sigma)$, we have that $\text{var}_m(\phi)\to 0$ as $m\to\infty$. Let $m\in\N$ be large enough such that $\text{var}_m(\phi)<\infty$. Define $\phi_m(x)=\sup\{\phi(y):y\in[x_1,\ldots, x_{m}]\}$, where $x=(x_1,x_2,\ldots)$. Note that $\phi_m$ is non-negative, $\text{var}_m(\phi_m)=0$, and $\|\phi-\phi_m\|_0\le \text{var}_m(\phi)$. In particular, $\|\phi-\phi_m\|_0\to 0$ as $m\to\infty$. Let $(C_i^m)_i$ be an enumeration of the set of cylinders of length $m$. Define $A^m_k=\bigcup_{i=1}^k C_i^m$. It follows that 
$$\liminf_{n\to\infty}\int \phi_md\mu_n\ge \lim_{n\to\infty}\int \phi_m 1_{A^m_k}d\mu_n=\int \phi_m 1_{A_k^m}d\mu.$$
Sending $k\to \infty$ we get that $\liminf_{n\to\infty}\int \phi_md\mu_n\ge  \int \phi_md\mu.$ Finally send $m\to\infty$ and use that $|\int \phi d\nu-\int \phi_m d\nu|\le \|\phi-\phi_m\|_0$, for every $\nu\in\M_{\le 1}(\sigma)$.
\end{proof}

\subsection{Entropy at infinity}\label{sec:entinf} Let $(\Sigma,\sigma)$ be a transitive CMS with alphabet $\N$. Fix $a\in \N$. Let $\text{Per}_a(n)=\{x\in[a]:\sigma^n(x)=x\}$. We define the topological entropy (or entropy) of $(\Sigma,\sigma)$ by the formula $h_{top}(\sigma)=\limsup_{n\to\infty} \frac{1}{n}\log \# \text{Per}_a(n),$ which is independent of $a$. Let
$$\text{Per}_a(q,M,n)=\big\{x\in  \text{Per}_a(n): \#\{k\in\{1,...,n\}:x_k\le q\}\le \frac{n}{M}\big\},$$ where $q,M\in\N$. Define
 $$\delta_\infty(M,q)=\limsup_{n\to\infty} \frac{1}{n}\log \# \text{Per}_a(q,M,n), \quad\text{and}\quad \delta_\infty=\inf_{q}\inf_{M} \delta_\infty(M,q).$$
We refer to $\delta_\infty$ as the topological entropy at infinity of $(\Sigma,\sigma)$. It follows from the transitivity of $(\Sigma,\sigma)$ that $\delta_\infty$ is independent of $a$. Note that $\text{Per}(q,M,n)\subseteq \text{Per}(q',M',n)$, whenever $M\ge M'$ and $q\ge q'$. In particular, we have that
$\delta_\infty=\lim_{q\to\infty}\lim_{M\to\infty} \delta_\infty(M,q).$

Define 
$$h_\infty=\sup_{(\mu_n)_n\to 0}\limsup_{n\to\infty}h_{\mu_n}(\sigma),$$
where the supremum runs over all sequences $(\mu_n)_n$ converging on cylinders to the zero measure. If there is no such sequence we set $h_\infty=-\infty$ (we consider this definition in order to be consistent with the definition of the pressure at infinity).  The quantity $h_\infty$ is called the measure theoretic entropy at infinity of $(\Sigma,\sigma)$. It is proved in \cite[Theorem 1.4]{itv} that if $(\Sigma,\sigma)$ has finite  entropy, then $h_\infty=\delta_\infty$. We refer to this property as the variational principle for the entropies at infinity. 

\begin{remark}\label{rem:infent} If $(\Sigma,\sigma)$ has infinite topological entropy, then $\delta_\infty=\infty$ (see \cite[Proposition 8.3]{itv}). If there exists a sequence of invariant probability measures that converges to the zero measure, then $h_\infty=\infty$ (see Lemma \ref{lem:infpre}), otherwise $h_\infty=-\infty$. In the infinite entropy setting it is possible that there is no sequence in $\M(\sigma)$ which converges on cylinders to the zero measure (see \cite[Example 4.17]{iv}). We will discuss this issue in Section \ref{sec:spr}. \end{remark}

The entropy at infinity is related to the upper semi-continuity of the entropy map when escape of mass is allowed, as we can see in the next result. 

\begin{theorem}[{\cite[Theorem 1.1]{itv}}] \label{itv} Let $(\Sigma,\sigma)$ be a transitive CMS with finite entropy.  Let $(\mu_n)_{n}$ be a sequence in $\M(\sigma)$ which converges on cylinders to $\lambda \mu$, where $\mu\in \M(\sigma)$ and $\lambda\in [0,1]$. Then, $$\limsup_{n\to \infty} h_{\mu_n}(\sigma)\le \lambda h_{\mu}(\sigma)+(1-\lambda)\delta_\infty.$$
Moreover, the inequality is sharp.
\end{theorem}

\subsection{Thermodynamic formalism}\label{introcms} Let $(\Sigma,\sigma)$ be a transitive CMS. Fix $a\in S$. 
For a potential $\phi:\Sigma\to \R,$ we define $$Z_n(\phi,a)=\sum_{\sigma^n x=x} e^{S_n \phi(x)}1_{[a]}(x), \text{ and } P^{top}(\phi,a)=\limsup_{n\to \infty} \frac{1}{n}\log Z_n(\phi,a).$$
If $P^{top}(\phi, a)$ is independent of $a \in S$, we refer to this value as the topological pressure, or Gurevich pressure, of $\phi$. The topological pressure of $\phi$ is denoted by $P^{top}(\phi)$. If $\phi$ has summable variations, then the topological pressure is well defined. Furthermore, if $(\Sigma,\sigma)$ is  mixing then the limsup is a limit (see \cite[Theorem 1]{sa1}). Note that the topological pressure of the zero potential is the topological entropy of $(\Sigma,\sigma)$.

Define
 $$Z_n^*(\phi,a)=\sum_{\sigma^nx=x} e^{S_n\phi(x)}1_{\psi_a=n}(x),\text{ and } P^*(\phi,a)=\limsup_{n\to\infty}\frac{1}{n}\log Z_n^*(\phi,a),$$ where $\psi_a(x)=1_{[a]}(x)\inf\{n\ge1:\sigma^n x\in [a]\}$. Extending notions studied for Markov chains, Sarig \cite{sa1} classified potentials according to their recurrence properties.
\begin{definition} Let $(\Sigma,\sigma)$ be a transitive CMS and $\phi$ a  potential with summable variations and finite pressure. Set $\lambda=\exp P^{top}(\phi)$.
\begin{enumerate}
\item If $\sum_{n\ge 1}\lambda^{-n}Z_n(\phi,a)$ diverges we say that $\phi$ is {recurrent}.
\item If $\sum_{n\ge 1}\lambda^{-n}Z_n(\phi,a)$ converges we say that $\phi$ is {transient}.
\item If $\phi$ is recurrent and $\sum_{n\ge 1}n\lambda^{-n}Z^*_n(\phi,a)$ converges we say that $\phi$ is {positive recurrent}.
\item If $\phi$ is recurrent and $\sum_{n\ge 1}n\lambda^{-n}Z^*_n(\phi,a)$ diverges we say that $\phi$ is {null recurrent}.
\end{enumerate}
\end{definition}

The transfer operator $L_\phi$ associated with $\phi$ is given by $$L_\phi f(x)=\sum_{\sigma y=x} e^{\phi(y)}f(y),$$ which also induces a dual operator $L_\phi^*$ acting on the space of Borel measures. Sarig \cite{sa1,sa3} proved that $\phi$ is recurrent if and only if there exists a positive continuous function $h$ and a conservative measure $\nu$ finite and positive on cylinders such that $L_\phi h=e^{P^{top}(\phi)}h,$ and $L^*_\phi \nu=e^{P^{top}(\phi)}\nu$.  The measure $hd\nu$ is the Ruelle-Perron-Frobenius measure (RPF measure) associated to $\phi$.  Moreover, $\phi$ is positive recurrent if and only if $hd\nu$ is a finite measure. In this case we define $\mu_\phi=(\int hd\nu)^{-1}hd\nu$. If $\phi$ is positive recurrent and $\int \phi d\mu_\phi>-\infty$, then $\mu_\phi$ is an equilibrium state of $\phi$. A fundamental result in this setting is that a potential with summable variations has at most one equilibrium state (see \cite[Theorem 1.2]{bs}). On the other hand, if $\phi$ has an equilibrium state, then $\phi$ is positive recurrent and $\mu_\phi$ is the unique equilibrium state of $\phi$. 

\begin{example}\label{ex:full} Let $(\N^\N,\sigma)$ be the full shift on a countable alphabet and $\phi:\N^\N\to\R$ a potential such that $\text{var}_1(\phi)=0$. In this case $\phi(x)=a(x_1)$, where $x_1$ is the first coordinate of $x\in \N^\N$, and $a:\N\to\R$ is a function. Then,
$P(\phi)=\log\big(\sum_{i\in \N}e^{a(i)}\big).$
\end{example}

\subsubsection{Variational principle} The measure-theoretic pressure of $\phi$ is defined by the formula
\begin{align*} P(\phi)=\sup_{\mu\in \M_\phi(\sigma)}\bigg(h_\mu(\sigma)+\int \phi d\mu\bigg),\end{align*}
where $\M_\phi(\sigma)=\{\mu\in\M(\sigma): \int \phi d\mu>-\infty\}$. If $(\Sigma,\sigma)$ is mixing and $\phi$ has summable variations, then $P^{top}(\phi)=P(\phi)$ (see \cite[Theorem 3]{sa1} and \cite[Theorem 2.10]{ijt}). This statement is known as the variational principle for the pressure. In the following remark, we discuss how to deduce the variational principle for transitive CMS from the mixing case, both for completeness and for later use.

\begin{remark}\label{mixtotra}Let $(\Sigma, \sigma)$ be a transitive CMS with alphabet $S$ and period $p$, that is, $p$ is the greatest common divisor of the set of periods of periodic orbits in $\Sigma$. Assume $p>1$ (a CMS is mixing if and only if $p=1$). Let us define an equivalence relation on $S$. We say that $a\sim b$ if $a=b$, or if there exists an admissible word  $aq_1\ldots q_{n-1} b,$ where $n$ is a multiple of $p$.  There are $p$ equivalence classes $S_1,\ldots, S_p$ in the alphabet $S$. Define $\Sigma_i=\{x\in\Sigma: x_1\in S_i\}$. It follows that $\Sigma=\bigcup_{i=1}^p \Sigma_i$, $\sigma(\Sigma_i)=\Sigma_{i+1},$ and $\sigma^p(\Sigma_i)=\Sigma_i$ (maybe after reordering the equivalent classes). The dynamical system $(\Sigma_i,\sigma^p)$ is a CMS whose alphabet is given by the set of admissible words $\{\tilde{a}=a_1\ldots a_{p}: a_1\in S_i\}$, where $\tilde{a}\tilde{b}$ is admissible if and only if $a_{p}b_1$ is admissible in $\Sigma$.  The period of  $(\Sigma_i,\sigma^p)$ is equal to one, in other words, it is mixing. The partition $\{\Sigma_i\}_{i=1}^p$ is called the spectral decomposition of $\Sigma$.

Note that if we induce on $\Sigma_1$, the first return time is constant and equal to $p$. Let $\overline{\phi}=S_p\phi$, be the  induced potential on $\Sigma_1$, and for   $\mu\in \M(\sigma)$, we set $\bar{\mu}(\cdot)=\mu(\Sigma_1\cap \cdot)/\mu(\Sigma_1)$ (note that $\mu(\Sigma_1)> 0$, for every $\mu\in\M(\sigma)$). The map $\mu\mapsto\bar{\mu}$ is a bijection between $\M(\sigma)$ and $\M(\Sigma_1,\sigma^p)$ with  inverse  $\nu\mapsto \frac{1}{p}\sum_{i=0}^{p-1}\sigma_*^i\nu$. It follows by Kac's and Abramov's formula that $h_{\bar{\mu}}(\sigma^p)+\int \overline{\phi}d\bar{\mu}=p(h_\mu(\sigma)+\int \phi d\mu)$, for every $\mu\in \M(\sigma)$. In particular $P(\overline{\phi})=pP(\phi)$. On the other hand, $Z_n(\phi,a)$ is different from zero only if $n$ is multiple of $p$, therefore 
$$P^{top}(\phi)=\limsup_{k\to\infty}\frac{1}{kp}\log Z_{kp}(\phi,a).$$
Choose $a\in S_1$, then
$$P^{top}(\phi)=\limsup_{k\to\infty}\frac{1}{kp}\log Z_{kp}(\phi,a)=\frac{1}{p}\limsup_{k\to\infty}\frac{1}{k}\log Z_{k}(\overline{\phi},a)=\frac{1}{p}P^{top}(\overline{\phi}).$$
Finally, since $(\Sigma_1,\sigma^p)$ is mixing we have that $P(\overline{\phi})=P^{top}(\overline{\phi})$, and therefore $P(\phi)=P^{top}(\phi)$.
\end{remark}

\begin{remark}
Observe  that $|P^{top}(\phi,a)-P^{top}(\phi',a)|\le \|\phi-\phi'\|_0$, and that $|P(\phi)-P(\phi')|\le \|\phi-\phi'\|_0$. In particular, the pressure functions $\phi\mapsto P(\phi)$ and $\phi\mapsto P^{top}(\phi,a)$ are continuous with respect to the $C^0-$norm. Since potentials with summable variations are dense in $\H$, we conclude that the topological pressure is well defined for potentials in $\H$ and that the variational principle still holds. 
\end{remark}

\subsection{A recoding}\label{recode}  Let $(\Sigma,\sigma)$ be a CMS with alphabet $S$ and transition matrix $M$.  In this subsection, we describe a recoding of $(\Sigma,\sigma)$, which will be used in subsection \ref{subsec:51}. Fix $m\in \N$. Define $S_m=\{(a_1,\ldots,a_m)\in S^m: \prod_{i=1}^{m-1}M_{a_i,a_{i+1}}=1\},$ and the $S_m\times S_m$ incidence matrix $M_m$, given by $(M_m)_{(a_1,\ldots,a_m),(b_1,...,b_m)}=1$, if $a_{i+1}=b_i$, for all $1\le i\le m-1$, and zero otherwise. The CMS with alphabet $S_m$ and transition matrix $M_m$ is denoted by $(\Sigma_m,\sigma_m)$.

 Let $\pi_m:\Sigma_m\to \Sigma$ be the canonical map that sends 
 $x=((x_1,\ldots,x_{m}),(x_2,\ldots,x_{m+1}),\ldots)$ to $\pi_m(x)=(x_1,x_2,\ldots)$.  It follows by construction of $\Sigma_m$ that $\pi_m$ is a bijection and that $\pi_m\circ \sigma_m=\sigma\circ \pi_m$. Observe  that
\begin{align}\label{pim}\pi_m([(a_1,\ldots,a_m),(a_2,\ldots,a_{m+1}),\ldots,(a_{k},\ldots,a_{k+m-1})])=[a_1,\ldots,a_{k+m-1}],\end{align}
for every $k\in \N$. Note that $\pi_m$ maps a basis of the topology of $\Sigma_m$ to a basis of the topology of $\Sigma$. In particular, $\pi_m$ is a homeomorphism and then $\pi_m$ is a topological conjugacy between $(\Sigma_m,\sigma_m)$ and $(\Sigma,\sigma)$. It follows that the push forward $(\pi_m)_*:\M(\sigma_m)\to \M(\sigma)$ is a homeomorphism with respect to the weak$^*$ topologies. The map $(\pi_m)_*$ induces a bijection between  $\M_{\le 1}(\sigma_m)$ and $\M_{\le1}(\sigma)$, which is not necessarily a homeomorphism with respect to the cylinder topologies (see Example \ref{counter}). For $\mu\in \M_{\le1}(\sigma),$ we set 
$\tilde{\mu}=(\pi_m)_*^{-1}\mu\in \M_{\le 1}(\sigma_m)$.

\begin{remark}\label{rem:topm} Note that $\pi_m$ maps cylinders of length $k$ in $\Sigma_m$ to cylinders of length $m+k-1$ in $\Sigma$. In particular, if $(\mu_n)_n$ is a sequence in $\M(\sigma)$ that converges on cylinders to $\mu\in\M_{\le1}(\sigma)$, then $(\tilde{\mu}_n)_n$ converges on cylinders to $\tilde{\mu}$. On the other hand, if $(\tilde{\mu}_n)_n$ is a sequence in $\M(\sigma_m)$  that converges on cylinders to $\tilde{\mu}\in\M_{\le1}(\sigma_m)$, then $\lim_{n\to\infty}\mu_n(C)=\mu(C)$, for all cylinder $C\subseteq \Sigma$ of length greater than $m$.
\end{remark}

\begin{example}\label{counter} Consider the full shift $(\N^\N,\sigma)$ and the periodic points $p_n=\overline{1n}$. Each periodic point $p_n$ defines a periodic measures $\mu_n=\frac{1}{2}(\delta_{p_n}+\delta_{\sigma(p_n)})$. Observe that $(\mu_n)_n$ satisfies that $\mu_n([1])=\frac{1}{2}$, and  $\lim_{n\to\infty}\mu_n([a,b])=0$, for every $a,b\in \N$. Consider the construction above for $m=2$, and note that the sequence $(\tilde{\mu}_n)_n$ converges on cylinders to the zero measure, but $(\mu_n)_n$ does not converge to a countably additive measure. In Remark \ref{rem:preclaim}, we provide a condition that rules out this type of behavior.
\end{example}

\section{Pressure at infinity and the variational principle}\label{secpre}
In this section, we study the pressure at infinity for potentials defined on countable Markov shifts. We consider two versions of the pressure at infinity: one based on measure-theoretic data and the other on topological data. We will prove Theorem \ref{varpriplus}, which states that both versions coincide for potentials with summable variations and finite pressure. To prove the variational principle for the pressures at infinity, we first derive a formula for the topological pressure at infinity in terms of the topological pressure (see Proposition \ref{inf}). Finally, in the last subsection, we introduce some definitions and prove several useful results needed in later sections.

\subsection{Topological pressure at infinity}\label{subsec:presinf}
Let $(\Sigma,\sigma)$ be a transitive CMS with alphabet $\N$. Fix $a\in \N$. Recall that $\text{Per}_a(n)=\{x\in\Sigma:x_1=a,\text{ and }\sigma^n(x)=x\}$, and 
$\text{Per}_a(q,M,n)=\{x\in  \text{Per}_a(n): \#\{k\in\{1,...,n\}:x_k\le q\}\le \frac{n}{M}\},$ where $q,M\in\N$.

\begin{definition} Given a potential $\phi:\Sigma\to\R,$ and $q,M\in\N$, we define $$Z_n(\phi,a; q,M)=\sum_{x\in \text{Per}_a(q,M,n)}\exp(S_n\phi(x)).$$
 If $\text{Per}_a(q,M,n)=\emptyset$, we set $Z_n(\phi,a;q,M)=0$.  Define
$$P^{top}_\infty(\phi,a;q,M)=\limsup_{n\to\infty} \frac{1}{n}\log Z_n(\phi,a;q,M),\quad\text{and}\quad P^{top}_\infty(\phi,a)=\inf_q\inf_M P_\infty(\phi,a;q,M).$$
If $P^{top}_\infty(\phi,a)$ is independent of $a,$ we refer to this value as the topological pressure at infinity of $\phi$, which we denote by $P^{top}_\infty(\phi)$.
\end{definition}

Since $\text{Per}_a(q,M,n)\subseteq \text{Per}_a(q',M',n)$, whenever $M\ge M'$ and $q\ge q'$, it follows that
$$P_\infty(\phi,a)=\lim_{q\to\infty}\lim_{M\to\infty} P_\infty(\phi,a;q,M).$$

In the next proposition we establish a useful formula for the topological pressure at infinity of a potential. Let $\text{Per}_a(q,M,n)^c=\text{Per}_a(n)\setminus \text{Per}_a(q,M,n).$ 

\begin{proposition}\label{inf} Let $\phi:\Sigma\to\R$ be a potential with summable variations. Let $V=\sum_{k=1}^\infty \frac{1}{k} 1_{[k]}\in C_0(\Sigma)$. Then, $$P^{top}_\infty(\phi,a)=\lim_{t\to\infty} P(\phi-tV).$$
\end{proposition}
\begin{proof} We will first prove that $P^{top}_\infty(\phi,a)=P^{top}_\infty(\phi-tV,a)$, for every $t\ge 0$. Observe that if $x\in \text{Per}_a(q,M,n)$, then $S_nV(x)\le \frac{n}{M}+\frac{n}{q}$. It follows that,
$$Z_n(\phi-tV,a;q,M)\ge Z_n(\phi,a;q,M)\exp\bigg(-nt\bigg(\frac{1}{M}+\frac{1}{q}\bigg)\bigg),$$
and therefore,
$$P^{top}_\infty(\phi-tV,a;q,M)\ge P^{top}_\infty(\phi,a;q,M)-t\bigg(\frac{1}{M}+\frac{1}{q}\bigg).$$
Sending $q$ and $M$ to infinity we obtain $P^{top}_\infty(\phi-tV,a)\ge P^{top}_\infty(\phi,a)$. The other inequality follows from $\phi\ge \phi-tV$. We conclude that $P_\infty^{top}(\phi,a)=P_\infty^{top}(\phi-tV,a)$. 

Since $P^{top}_\infty(\phi,a)=P^{top}_\infty(\phi-tV,a)\le P(\phi-tV)$, we have $P^{top}_\infty(\phi,a)\le \lim_{t\to\infty}P(\phi-tV)$. In order to complete the proof, we need to show that  $\lim_{t\to\infty}P(\phi-tV)\le P^{top}_\infty(\phi,a)$. 

Fix $q,M\in\N$. If $x\in \text{Per}_a^c(q,M,n)$, then $\#\{1\le k\le n:x_k\le q\}> \frac{n}{M}$. In particular, $S_nV(x)> \frac{n}{Mq}.$ Then,
$$\sum_{x\in \text{Per}_a^c(q,M,n)}\exp(S_n(\phi-tV)(x))\le \exp\bigg(-\frac{tn}{Mq}\bigg)\sum_{x\in \text{Per}_a^c(q,M,n)}\exp(S_n\phi(x)).$$
By definition of the pressure, there exists $n_0\in\N$ such that if $n\ge n_0,$ then \\$\sum_{x\in \text{Per}_a^c(q,M,n)}\exp(S_n\phi(x))\le \exp(n(P(\phi)+1)).$ Therefore,
\begin{align*}\sum_{x\in \text{Per}_a^c(q,M,n)}\exp(S_n(\phi-tV)(x))\le \exp\bigg(n\bigg(P(\phi)+1-\frac{t}{Mq}\bigg)\bigg),\end{align*}
for every $n\ge n_0$. Finally, for $n\ge n_0$ and $t>qM(P(\phi)+1+\kappa)$, we have
\begin{align*} Z_n(\phi-tV,a)&=\sum_{x\in \text{Per}_a(q,M,n)}\exp(S_n(\phi-tV)(x))+\sum_{x\in \text{Per}_a^c(q,M,n)}\exp(S_n(\phi-tV)(x))\\
&\le \sum_{x\in \text{Per}_a(q,M,n)}\exp(S_n\phi(x))+e^{-\kappa n},
 \end{align*}
for every $\kappa>0$. We conclude that $P(\phi-tV)\le \max\{P^{top}_\infty(\phi,a;q,M),-\kappa\}$, and therefore  $\lim_{t\to\infty}P(\phi-tV)\le P^{top}_\infty(\phi,a;q,M)$. 
\end{proof}

\begin{remark} It follows from Proposition \ref{inf} that if $\phi:\Sigma\to\R$ has summable variations, then $P^{top}_\infty(\phi,a)$ is independent of $a\in\N$. In particular, $P^{top}_\infty(\phi)$ is well defined. This can also be proved directly, similar to the proof of the independence of the base symbol for the topological pressure.
\end{remark}

\subsection{Measure-theoretic pressure at infinity and the variational principle}\label{sub:var}

In the introduction, we defined the {measure-theoretic pressure at infinity} as
\begin{align*}P_\infty(\phi)=\sup_{(\mu_n)_n\to 0}\limsup_{n\to\infty}\bigg(h_{\mu_n}(\sigma)+\int \phi d\mu_n\bigg),\end{align*}
where the supremum runs over all sequences $(\mu_n)_n$ in $\M_{\phi}(\sigma)$ which converge on cylinders to the zero measure. If there is no such sequence, we set $P_\infty(\phi)=-\infty$. The next two results state some basic properties of the measure-theoretic pressure at infinity.

\begin{lemma}\label{convex}  The following holds:
\begin{enumerate}
\item $P_\infty(\phi )\le P(\phi )$
\item $P_\infty(\phi +c)=P_\infty(\phi )+c$, for every $c\in \R$. 
\item If $\phi \le \psi$, then $P_\infty(\phi )\le P_\infty(\psi)$.
\item $|P_\infty(\phi )-P_\infty(\psi)|\le \|\phi -\psi\|_0$.
\item The function $t\mapsto P_\infty(t\phi )$ is convex whenever finite.
\item $P_\infty(\phi +g)=P_\infty(\phi )$, whenever $g\in C_0(\Sigma)$. 
\end{enumerate}
\end{lemma}
\begin{proof} Parts (1)--(5) follow easily from the definition. For (6) note that $\lim_{n\to\infty} \int gd\mu_n=0$, for every sequence $(\mu_n)_n$ that converges on cylinders to the zero measure. 
\end{proof}

\begin{lemma}\label{lem:infpre} Suppose that there exists a sequence in $\M(\sigma)$ which converges on cylinders to the zero measure. Let $\phi\in C(\Sigma)$ be a potential such that $P(\phi)=\infty$. Then, $P_\infty(\phi)=\infty$. 
\end{lemma}
\begin{proof}
Since periodic measures are dense in $\M(\sigma)$, there exists a sequence $(\nu_n)_n$  in $\M_\phi(\sigma)$ that converges on cylinders to the zero measure. By the variational principle, there exists a sequence $(\mu_n)_n$ in $\M_\phi(\sigma)$ such that $h_{\mu_n}(\sigma)+\int \phi d\mu_n\ge n^2$. Define $\eta_n=(1-\frac{1}{n})\nu_n+\frac{1}{n}\mu_n$. The sequence $(\eta_n)_n$ converges to the zero measure and satisfies that $h_{\eta_n}(\sigma)+\int \phi d\eta_n\ge n$.
\end{proof}

The next proposition is the key ingredient in order to prove the variational principle for the pressures at infinity. 

\begin{proposition}\label{presinf} Let $\phi:\Sigma\to\R$ be a potential with summable variations and finite pressure. Let $(\mu_n)_n$ be a sequence in $\M_\phi(\sigma)$ which converges on cylinders to the zero measure. Then, \begin{align}\label{y1}\limsup_{n\to\infty} \bigg(h_{\mu_n}(\sigma)+\int \phi d\mu_n\bigg)\le P^{top}_\infty(\phi ).\end{align}
If $P^{top}_\infty(\phi )>-\infty$, then there exists a sequence of invariant probability measures, as described above, for which equality holds.
\end{proposition}

\begin{proof} Let  $V=\sum_{k=1}^\infty \frac{1}{k}1_{[k]}$. Note that $\lim_{n\to\infty}\int Vd\mu_n=0$. By the standard variational principle we have that 
$$\limsup_{n\to\infty}\bigg(h_{\mu_n}(\sigma)+\int\phi d\mu_n\bigg)=\limsup_{n\to\infty}\bigg(h_{\mu_n}(\sigma)+\int (\phi -tV)d\mu_n\bigg)\le P(\phi -tV).$$
In particular, the left-hand side is bounded from above by $\lim_{t\to\infty}P(\phi -tV)$. Inequality (\ref{y1}) follows from Proposition \ref{inf}.

Let us suppose now that $M:=P^{top}_\infty(\phi )>-\infty$. We will prove that inequality (\ref{y1}) is sharp. It follows from Proposition \ref{inf} that $P(\phi -tV)\ge M$, holds for all $t\in\R$. The variational principle guarantees the existence of $\sigma$-invariant probability measures $(\nu_n)_n$ such that 
\begin{align}\label{0} h_{\nu_n}(\sigma)+\int (\phi-nV) d\nu_n\ge P(\phi -nV)-\frac{1}{n},\end{align}
where $\int (\phi -nV)d\nu_n>-\infty$. In particular, $\int \phi d\nu_n>-\infty$. We claim that $(\nu_n)_n$ converges on cylinders to the zero measure. We argue by contradiction and suppose that  $\limsup_{n\to\infty}\nu_n(C)>0$, for some cylinder $C\subseteq \Sigma$. In this case we would have that $\limsup_{n\to\infty}n\int Vd\nu_n=\infty$, which is not possible because inequality (\ref{0}) implies that $P(\phi)-n\int Vd\nu_ n\ge M-1$. On the other hand, 
\begin{align*} P^{top}_\infty(\phi )=\lim_{n\to\infty}P(\phi -nV) &\le \liminf_{n\to\infty}\big(h_{\nu_n}(\sigma)+\int \phi d\nu_n-n\int Vd\nu_n\big)\\
&\le \liminf_{n\to\infty} \big(h_{\nu_n}(\sigma)+\int \phi d\nu_n\big).\end{align*}
In conclusion, $(\nu_n)_n$ is a sequence in $\M_\phi(\sigma)$ that converges on cylinders to the zero measure and $P^{top}_\infty(\phi )=\lim_{n\to\infty} h_{\nu_n}(\sigma)+\int \phi d\nu_n.$
\end{proof}

\begin{proof}[Proof of Theorem \ref{varpriplus}] In Proposition \ref{presinf} we proved that $P_\infty(\phi)\le P^{top}_\infty(\phi)$, and that equality holds if $P^{top}_\infty(\phi)>-\infty$.  Observe that inequality $P_\infty(\phi)\le P^{top}_\infty(\phi)$ also implies that if $P^{top}_\infty(\phi)=-\infty$, then $P_\infty(\phi)=-\infty$. 
\end{proof}

\begin{remark}\label{rem:inf} Let $\phi$ be a potential with summable variations. Note that if $\phi$ has finite pressure and there is no sequence in $\M(\sigma)$ that converges to the zero measure, then $P_\infty(\phi)=P^{top}_\infty(\phi)=-\infty$ (see Proposition \ref{presinf}). Suppose now that $P(\phi)=\infty$. Then,
\begin{enumerate}
\item $P^{top}_\infty(\phi)=\infty$ (see Proposition \ref{inf}). Then, $P(\phi)<\infty$ if and only if $P_\infty^{top}(\phi)<\infty$.
\item If there exists a sequence in $\M(\sigma)$ which converges on cylinders to the zero measure,  then $P_\infty(\phi)=\infty$ (see Lemma \ref{lem:infpre}). If there is no sequence converging to the zero measure, then by definition $P_\infty(\phi)=-\infty$.
\end{enumerate}
\end{remark}

\begin{remark}\label{rem:Hinfty}
Observe  that $|P_\infty^{top}(\phi,a)-P_\infty^{top}(\phi',a)|\le \|\phi-\phi'\|_0$, and that $|P_\infty(\phi)-P_\infty(\phi')|\le \|\phi-\phi'\|_0$. In particular, the pressure at infinity functions $\phi\mapsto P_\infty(\phi)$ and $\phi\mapsto P_\infty^{top}(\phi,a)$ are continuous with respect to the $C^0-$norm. Since potentials with summable variations are dense in $\H$, we conclude that the topological  pressure at infinity is well defined for potentials in $\H$ and that if $\phi\in\H$, then $P_\infty^{top}(\phi)=P_\infty(\phi)$.
\end{remark}

\subsection{Further properties and classes of potentials}\label{sec:phi} In this subsection, we introduce some definitions and establish useful facts for later use.

 For a potential $\phi \in C(\Sigma)$ with finite pressure, we define $$s_\infty(\phi )=\inf\{t\in [0,\infty):P(t\phi )<\infty\}.$$
By definition $s_\infty(\phi)\in [0,1]$. Note that if $\|\phi_1-\phi_2\|_0<\infty$, then $s_\infty(\phi_1)=s_\infty(\phi_2)$. 

\begin{lemma}\label{lem:-inf} Let $\phi\in C(\Sigma)$ be a potential with finite pressure. Suppose that $s_\infty(\phi)<1$. 
Let $(\mu_n)_n$ be sequence in $\M_\phi(\sigma)$ such that $\lim_{n\to\infty}\int \phi d\mu_n=-\infty$. Then, $$\lim_{n\to\infty}h_{\mu_n}(\sigma)+\int \phi d\mu_n=-\infty.$$
\end{lemma}
\begin{proof} Since the pressure is convex whenever finite, it follows that $P(t\phi)<\infty$ for every $t\in (s_\infty(\phi),1]$. Consider $\epsilon>0$ small such that $P((1-\epsilon)\phi)<\infty$, and note that $$h_{\mu_n}+\int \phi d\mu_n\le P((1-\epsilon)\phi)+\epsilon\int\phi d\mu_n,$$
for every $n\in\N$.
\end{proof}

\begin{lemma}\label{prop:gapr} Let $\phi\in C(\Sigma)$ be a potential with finite pressure. Assume that $s_\infty(\phi)<1$, and that $P_\infty(\phi)>-\infty$. Let $(\mu_n)_n$ be a sequence in $\M_\phi(\sigma)$ that converges on cylinders to the zero measure and such that $\lim_{n\to\infty}\big(h_{\mu_n}(\sigma)+\int \phi d\mu_n\big)=P_\infty(\phi)$. Therefore, $\liminf_{n\to\infty}\int \phi d\mu_n>-\infty$.
\end{lemma}
\begin{proof} We argue by contradiction and assume that $\liminf_{n\to\infty}\int \phi d\mu_n=-\infty$. In particular, there exists a subsequence whose limit is $-\infty$. The result now follows from Lemma \ref{lem:-inf}.
\end{proof}

In the introduction, we defined  $$\Psi=\{\phi\in C_{uc}(\Sigma): \inf \phi>0, \text{var}_2(\phi)<\infty, P(-\phi)<\infty\}.$$
Note that if $\phi\in\Psi$, then $t\mapsto P(-t\phi)$ is  strictly decreasing, and therefore $P(-t\phi)$ is finite for  $t\in (s_\infty(-\phi),\infty)$, and $P(-t\phi)=\infty$, if $t\in (0,s_\infty(-\phi))$. Furthermore, $-t\phi\in \H$ for every $t>0$, and therefore Remark \ref{rem:Hinfty} applies.

\begin{definition}\label{def:123} Let $\phi \in \Psi$. Define $$h(\phi)=\inf\{t:P^{top}_\infty(-t\phi)\le 0\}.$$ We say that $\phi $ belongs to $\Psi_1$ if there exists $s>0$ such that $P^{top}_\infty(-s\phi)\in (0,\infty)$. We say that $\phi\in \Psi_2$ if $\phi$ does not belong to $\Psi_1$. 
\end{definition}

\begin{remark}\label{rem:inutil} Note that $P(-t\phi)<\infty$  if and only if $P_\infty^{top}(-t\phi)<\infty$ (see Remark \ref{rem:inf}(1)). In particular, $h(\phi)\ge s_\infty(-\phi)$. Moreover, if $t>s_\infty(-\phi)$, then $P_\infty^{top}(-t\phi)=P_\infty(-t\phi)$ (see Remark \ref{rem:Hinfty}). Furthermore:
\begin{enumerate}
\item If $\phi\in \Psi_1$, then there exists a unique real number $s_\phi>s_\infty(-\phi)$ such that \\ $P^{top}_\infty(-s_\phi\phi)=0$. In this case $h(\phi)=s_\phi$. 
\item If $\phi\in \Psi_2$, then $P^{top}_\infty(-s\phi)<0,$ for every $s> s_\infty(-\phi)$. In this case $h(\phi)=s_\infty(-\phi)$.
\end{enumerate}
\end{remark}

The quantity $h(\phi)$ has a clear dynamical interpretation, as it represents the entropy at infinity of the suspension flow over $(\Sigma,\sigma)$ with roof function $\phi$ (see Section \ref{sec:flow}).

\section{SPR potentials and the existence of sequences converging to the zero measure}\label{sec:spr} In his study of phase transitions, Sarig \cite{sa2} introduced the class of strongly positive recurrent potentials (referred to as SPR). The thermodynamic formalism of SPR potentials is similar to that of H\"older potentials in subshifts of finite type. For instance, a weakly H\"older SPR potential  is positive recurrent and the RPF measure has exponential decay of correlations (see \cite{sa2, cs}). Furthermore, it is proven in \cite[Theorem 2.2]{cs} that the space of SPR potentials is open and dense in the space of weakly H\"older potentials with finite pressure (in the $C^0$ and finer topologies). 

Let $(\Sigma,\sigma)$ be a CMS with alphabet $S$ and transition matrix $M$. Let $a\in S$, and define $\Sigma(a)$ as the set of points in $[a]$ that return to $[a]$ infinitely many times. Let $\sigma_a:\Sigma(a)\to\Sigma(a)$ be the first return map to $\Sigma(a)$. The system $(\Sigma(a),\sigma_a)$ is conjugate to the full shift on a countable alphabet. The Markov partition on $\Sigma(a)$ is given by  cylinders of the form $[a,b_1,\ldots,b_n]$, where each $b_i\ne a,$ and $b_na$ is an admissible word. Denote by $\tau_a:\Sigma(a)\to \N$ the first return time. The function $\tau_a$ is locally constant and depends only on the first coordinate of the Markov structure on $\Sigma(a)$. For a potential $\psi:\Sigma\to\R$ we denote by $\overline{\psi}:\Sigma(a)\to\R$ the induced potential, that is, $\overline{\psi}(x)=\sum_{i=0}^{\tau_a(x)-1}\psi(\sigma^i x)$, for every $x\in \Sigma(a)$. Let $\Phi$ be the set of weakly H\"older potentials in $\Sigma$ that are bounded above and have finite pressure. 

\begin{remark}\label{rem:induce0}
If $\psi\in \Phi$ is recurrent and $P(\psi)=0$, then $\overline{\psi}$ is weakly H\"older, positive recurrent and $P(\overline{\psi})=0$ (see \cite[Lemma 3]{sa2}).
\end{remark}

Let  $\phi\in \Phi$. Define $p_a^*(\phi)=\sup\{t:P(\overline{\phi+t})<\infty\}.$ It is proved in \cite{sa2} (see also \cite[Section 7.3]{cs}) that $p_a^*(\phi)=-P^*(\phi,a)$ (see subsection \ref{introcms}). The discriminant of $\phi$ at $a\in S$ is defined as
\begin{align}\label{discriminant}\Delta_a(\phi)=\sup\{P(\overline{\phi+t}): t<p_a^*(\phi)\}.\end{align}

The following result was proved by Sarig in \cite[Theorem 2]{sa2}
\begin{theorem}[Discriminant Theorem] \label{disc} Let $(\Sigma,\sigma)$ be a mixing CMS and $\phi\in \Phi$. Let $a\in S$.
\begin{enumerate}
\item The equation $P(\overline{\phi+t})=0$ has a unique solution $t=p(\phi)$ if $\Delta_a(\phi)\ge 0$ and no solution if $\Delta_a(\phi)<0$. Moreover,
\[ P(\phi)= \left\{
\begin{array}{ll}
      -p(\phi) \text{ if }\Delta_a(\phi)\ge 0,\\
-p^*_a(\phi) \text{ if } \Delta_a(\phi)<0.
\end{array} 
\right. \]
\item $\phi$ is positive recurrent if $\Delta_a(\phi)>0$ and transient if $\Delta_a(\phi)<0$. If $\Delta_a(\phi)=0$, then $\phi$ is either positive recurrent or null recurrent. 
\end{enumerate}
\end{theorem}

We say that $\phi\in \Phi$ is strongly positive recurrent if $\Delta_a(\phi)>0$.  If the zero potential is SPR, we say that $(\Sigma,\sigma)$ is strongly positive recurrent. The class of SPR (or stable-positive recurrent) countable Markov shifts has been extensively studied by Gurevich and Savchenko in \cite{gs}. It was proved in \cite[Theorem 8.2]{rs} that if $\phi\in\Phi$, then $\phi$ is SPR if and only if $P_\infty(\phi)<P(\phi)$. It is natural to consider the following definition.
\begin{definition}\label{def:sprn} We say $\phi\in C(\Sigma)$ is SPR if $P(\phi)$ is finite and $P_\infty(\phi)<P(\phi)$.
\end{definition}

Let $G=(V,E)$ be the directed graph associated with $(\Sigma,\sigma)$. Recall that the set of vertices $V$ is identified with $S$ and $(i,j)\in E$ if and only if $M_{i,j}=1$. A subset $F\subseteq V$ is called a uniform Rome if there exists $N\in\N$ such that $V\setminus F$ has no paths in $G$ of length greater than $N$. A finite uniform Rome is a uniform Rome where $F$ is a finite set. In the next result, the equivalence between $(1)$ and $(2)$ was proved by Cyr in \cite[Theorem 2.1]{c}, and the equivalence between $(2)$ and $(3)$ follows from \cite[Theorem 2.3]{cs}.

\begin{theorem}\label{hj} Let $(\Sigma,\sigma)$ be a mixing CMS and $G$ the associated directed graph. Then the following statements are equivalent:
\begin{enumerate}
\item The graph $G$ has a finite uniform Rome.
\item The set of transient potentials in $\Phi$  is empty.
\item Every potential in $\Phi$ is SPR. 
\end{enumerate}
\end{theorem}

We add another statement to the equivalences mentioned above.

\begin{theorem}\label{cry}  Let $(\Sigma,\sigma)$ be a mixing CMS and $G$ the associated directed graph. Then the following statements are equivalent:
\begin{enumerate}
\item\label{a} The graph $G$ has a finite uniform Rome.
\item\label{b} The set of transient potentials in $\Phi$  is empty.
\item\label{c} $P_\infty^{top}(\phi)=-\infty$, for every $\phi\in \Phi$. In particular, every potential in $\Phi$ is SPR. 
\item\label{d} There is no sequence in $\M(\sigma)$ which converges on cylinders to the zero measure. 
\end{enumerate}
\end{theorem}

\begin{proof} Let $\phi\in\Phi$. It follows from Proposition \ref{presinf} that if $P_\infty^{top}(\phi)>-\infty$, then there are sequences of invariant probability measures that converge on cylinders to the zero measure. We conclude that (\ref{d}) implies that $P_\infty^{top}(\phi)=-\infty$, and therefore $\phi$ is SPR. In other words, (\ref{d}) implies (\ref{c}). By Theorem \ref{hj}, to complete the proof,  it is enough to show that (\ref{a}) implies (\ref{d}). Let $F\subseteq V$ be a finite uniform Rome of $G=(V,E)$. It follows from the definition of uniform Rome that if $\mu\in \M(\sigma)$ is a periodic measure, then $\mu(\bigcup_{s\in F} [s]\big)\ge 1/(N+1)$. In particular, there are no sequences of periodic measures that converge on cylinders to the zero measure. Since periodic measures are dense in $\M(\sigma)$ we conclude that (\ref{d}) holds. 
\end{proof}

\begin{remark}Theorem \ref{cry} still holds if $(\Sigma,\sigma)$ is transitive (see Remark \ref{mixtotra}).
\end{remark}

\section{A modified countable Markov shift} \label{con}

Let $(\Sigma,\sigma)$ be a countable Markov shift with alphabet $S$ and transition matrix $M$. Let $\tau:\Sigma\to\N$ be a potential that only depends on the first two coordinates, that is, $\text{var}_2(\tau)=0$. The potential $\tau$ defines a function on $Q:=\{(a,b)\in S\times S:M_{a,b}=1\},$ which we still denote by $\tau$. With this data, we will construct a countable Markov shift $(\widehat{\Sigma},\widehat{\sigma})$ with alphabet  $\widehat{S}$ and transition matrix $\widehat{M}$ as follows:
\begin{enumerate}
\item For each $(a,b)\in Q,$ we consider a collection of symbols $\{c^{a,b}_1,\ldots,c^{a,b}_{\tau(a,b)}\}$, where $c^{a,b}_1=a$.  We assume that the symbols $(c^{a,b}_i)_{i,a,b}$, where $2\le i\le \tau(a,b)$ and $(a,b)\in Q,$  are pairwise different and not in $S$.  Set $\widehat{S}=\bigcup_{(a,b)\in Q} \{c^{a,b}_1,\ldots,c^{a,b}_{\tau(a,b)}\}$. 
\item If $i\in \{1,\ldots,\tau(a,b)-1\}$, we set $\widehat{M}_{c_i^{a,b}, j}=1$ if and only if $j=c_{i+1}^{a,b}$. Additionally, $\widehat{M}_{c^{a,b}_{\tau(a,b)},q}=1$ if and only if  $q=b$. 
\end{enumerate}

The countable Markov shift $(\widehat{\Sigma},\widehat{\sigma})$ can be described in simple terms by considering its associated directed graph. Let $G=(V,E)$ be the directed graph associated with $(\Sigma,\sigma)$. The function $\tau$ assigns a natural number to each edge $e\in E$. If $\tau(e)>1$, we subdivide the edge $e$ into $\tau(e)$ sub-edges, preserving the orientation.  If $\tau(e)=1,$ we do not modify the edge $e$. The resulting graph $\widehat{G}$ represents $(\widehat{\Sigma},\widehat{\sigma})$. Infinite paths in $\widehat{G}$ are in one-to-one correspondence with infinite paths in $G$, but the dynamics on the edges are slowed down according to the function $\tau$. Note that $(\widehat{\Sigma},\widehat{\sigma})$ is transitive if and only if $(\Sigma,\sigma)$ is transitive.

For $(a,b)\in Q$ we define the admissible word in $(\widehat{\Sigma},\widehat{\sigma})$, $$w_{a,b}=c^{a,b}_1\ldots c^{a,b}_{\tau(a,b)}.$$ Note that $w_{a,b}w_{p,q}$ is admissible if and only if $b=p$. We will use the words $(w_{a,b})_{(a,b)\in Q}$ to define points and cylinders in $\widehat{\Sigma}$. These points and cylinders are understood to have entries corresponding to those in the admissible words.

 Let $\A=\bigcup_{s\in S} [s]\subseteq \widehat{\Sigma}$, and set $\pi:\Sigma\to\A$, given by $\pi(a_1,a_2,a_3,\ldots)=(w_{a_1,a_2},w_{a_2,a_3},\ldots)$, which is well defined by construction of $\widehat{\Sigma}$.

\begin{lemma}\label{lem:conj} Let $T:\A\to\A$ be the induced transformation of $\widehat{\sigma}$ on $\A$. Then $\pi$ is a topological conjugacy between $(\Sigma,\sigma)$ and $(\A,T)$. 
\end{lemma}
\begin{proof} It follows by the definition of $\widehat{\Sigma}$ that every $x\in \A$ can be uniquely written as $x=(w_{a_1,a_2},w_{a_2,a_3},\ldots)$, for some $y=(a_1,a_2,\ldots)\in \Sigma$. Equivalently, the map $\pi$ is a bijection. Moreover, note that $T(x)=(w_{a_2,a_3},w_{a_3,a_4},\ldots)$, and therefore $\pi(\sigma(y))=T(\pi(y))$. We conclude that $\pi$ conjugates $(\Sigma,\sigma)$ and $(\A,T)$. It remains to prove that $\pi$ is a homeomorphism. Indeed, observe that $\pi([a_1,a_2,\ldots, a_n])=[w_{a_1,a_2},w_{a_2,a_3},\ldots, w_{a_{n-1},a_n}]$, and that the collection of cylinders of the form $[w_{a_1,a_2},w_{a_2,a_3},\ldots, w_{a_{n-1},a_n}]$ is a basis of the topology of $\widehat{\Sigma}$. Since $\pi$ maps a basis of $\Sigma$ to a basis of $\widehat{\Sigma}$, the claim follows.
\end{proof}

\begin{remark}\label{rem:discretesus}
The induced transformation $T:\A\to \A$ is given by $T|_{[a,c_{2}^{a,b}]}=\widehat{\sigma}^{\tau(a,b)}$. In other words, $\tau$ specifies the first return time to $\A$. The dynamical system $(\widehat{\Sigma},\widehat{\sigma})$ is a discrete suspension flow over $(\A,T)$, or equivalently, $(\Sigma,\sigma)$. The roof function, which depends only on the first two coordinates, is given by $\tau:\Sigma\to\N$, and the time-one map of the suspension flow corresponds to $\widehat{\sigma}$. In particular, there is a correspondence between $\sigma$-invariant measures on $\Sigma$ and $\widehat{\sigma}$-invariant measures on $\widehat{\Sigma}$. In the next subsection, we elaborate on this correspondence, which is explicit in this context, as well as on the relation between the cylinder topologies. We will use Kac's formula and Abramov's formula to relate the integral and entropy of invariant measures on these two CMS. 
\end{remark}

\subsection{Measures on $(\widehat{\Sigma},\widehat{\sigma})$ and convergence on cylinders}

\begin{lemma}\label{lem:p} Let $\nu$ be a $\widehat{\sigma}$-invariant sub-probability measure on $(\widehat{\Sigma},\widehat{\sigma})$ different from the zero measure. Then $\nu(\A)>0.$
\end{lemma}
\begin{proof} Since $\nu$ has positive mass, there exists a cylinder $[c_i^{a,b}]$ such that $\nu([c_i^{a,b}])$ is positive. If $i=1$ there is nothing to prove. Assume that $i>1$, in particular $\tau(a,b)\ge 2$. In this case $\sigma^{-(i-1)}[c_i^{a,b}]=[c_1^{a,b}, c_2^{a,b}]$, and therefore $\nu([a,c_2^{a,b}])=\nu([c_i^{a,b}])>0.$
\end{proof}

Let $\M(\sigma;\tau)=\{\mu\in\M(\sigma): \int\tau d\mu<\infty\}$. Define $\varphi:\M(\widehat{\sigma})\to \M(\sigma;\tau)$ by the formula $$\varphi(\nu)(E)=\frac{1}{\nu(\A)}\nu(\pi(E)).$$
The function $\varphi$ is the normalized restriction map to $\A$, where we have used the indentification between $\Sigma$ and $\A$ given by $\pi$. 
It follows from general results in ergodic theory 
that $\varphi(\nu)$ is a $\sigma$-invariant probability measure, and that it belongs to $\M(\sigma;\tau)$. Indeed, by Kac's formula we have that $\int {\tau} d \varphi(\nu)=\frac{1}{\nu(\A)}<\infty$, for every $\nu\in \M(\widehat{\sigma})$ (see Lemma \ref{lem:p}). In general, Kac's formula states that for $f\in L^1(\widehat{\Sigma},\nu)$, it holds that 
\begin{align}\label{eq:kac}\int_{\widehat{\Sigma}} fd\nu=\frac{1}{\int {\tau} d\varphi(\nu)}\int_\Sigma \sum_{k=0}^{{\tau}(x)-1}f(\widehat{\sigma}^k \pi(x))d\varphi(\nu)(x).\end{align}
The relationship between the measures $\nu$ and $\varphi(\nu)$ is explicit in terms of the measure of cylinders. Consider a cylinder in $\widehat{\Sigma}$, which by the structure of $\widehat{\Sigma}$ has either the form
\begin{align}\label{quepa1} C_1=[c_i^{a,b},\ldots ,c_{j}^{a,b}],\end{align} 
where $1\le i\le j\le \tau(a,b)$, or,
\begin{align}\label{quepa} C_2=[c_i^{a,p_1},\ldots ,c_{\tau(a,p_1)}^{a,p_1},w_{p_1,p_2}\ldots w_{p_{n-1},p_n}, c_1^{p_n,b},\ldots ,c_j^{p_n,b}],\end{align} 
where $1\le i\le \tau(a,p_1)$, and $1\le j\le \tau(p_n,b)$. Note that $D_1=\widehat{\sigma}^{-(i-1)}C_1=[w_{a,b}]$ (excluding the case when $i=j=1$, where $C_1=[a]$), and $D_2=\widehat{\sigma}^{-(i-1)}C_2=[w_{a,p_1},\ldots,w_{p_n,b}]$. It follows from Kac's formula (\ref{eq:kac}) for $f=1_{D_i}$ that 
\begin{align}\label{rel}\nu(C_1)=\frac{1}{\int {\tau} d\varphi(\nu)}\varphi(\nu)([a,b]),\text{ and } \nu(C_2)=\frac{1}{\int {\tau} d\varphi(\nu)}\varphi(\nu)([a,p_1,p_2,\ldots,p_n,b]).\end{align}
In particular, it also holds that
\begin{align}\label{rel2}\nu([a])=\sum_{b: M_{a,b}=1}\nu([a,c_1^{a,b}])=\frac{1}{\int {\tau} d\varphi(\nu)}\sum_{b: M_{a,b}=1} \varphi(\nu)([a,b])=\frac{1}{\int {\tau} d\varphi(\nu)}\varphi(\nu)([a]).\end{align} Equations (\ref{rel}) and (\ref{rel2}) explicitly demonstrate that $\varphi$ is a bijection. Indeed, for $\mu\in \M(\sigma;\tau)$ we define 
\begin{align*}\widehat{\mu}(C):=\frac{1}{\int {\tau} d\mu}\mu([a,p_1,\ldots ,p_n,b]),\end{align*}
 where $C=C_i\subseteq \widehat{\Sigma}$ as in  (\ref{quepa1}) or  (\ref{quepa}). We thus obtain a $\widehat{\sigma}$-invariant probability measure on $\widehat{\Sigma}$ (it defines a measure by Kolmogorov extension theorem, and its invariance can be verified on cylinders). Note that the map $\mu\mapsto\widehat{\mu}$ is the inverse of $\varphi$. 

\begin{remark}\label{rem:cyl} It follows from equations (\ref{rel}) and (\ref{rel2}) that for every cylinder $C\subseteq \widehat{\Sigma}$, there exists a cylinder $D\subseteq\Sigma$ such that $\widehat{\mu}(C)=\frac{1}{\int \tau d \mu}\mu(D)$, for every measure $\mu\in \M(\sigma;\tau)$. Analogously, given a cylinder $D\subseteq \Sigma$, there exists a cylinder $C\subseteq \widehat{\Sigma}$ such that $\widehat{\mu}(C)=\frac{1}{\int \tau d \mu}\mu(D)$, for every measure $\mu\in \M(\sigma;\tau)$.
\end{remark}

In the next subsection, we will calculate the entropy at infinity of $(\widehat{\Sigma}, \widehat{\sigma})$ in relation to quantities defined on $\Sigma$. The next two lemmas will help to establish this connection. The first lemma is a direct consequence of Remark \ref{rem:cyl}.

\begin{lemma}\label{tozero} Let $(\mu_n)_n$ be a sequence in $\M(\sigma;\tau)$. 
\begin{enumerate} 
\item\label{pi} Assume that $\limsup_{n\to\infty}\int \tau d\mu_n<\infty$. Then $(\widehat{\mu}_n)_n$ converges on cylinders to the zero measure if and only if $(\mu_{n})_n$ converges on cylinders to the zero measure.
\item\label{pii} If  $\lim_{n\to\infty}\int {\tau} d\mu_{n}=\infty$, then $(\widehat{\mu}_n)_n$ converges on cylinders to the zero measure.
\end{enumerate}
\end{lemma}

\begin{lemma}\label{topology} Let $(\mu_n)_n$ be a sequence in $\M(\sigma;\tau)$ which converges on  cylinders to $\lambda_1 \mu$,  where $\mu\in\M(\sigma;\tau),$ and $\lambda_1\in [0,1]$. Assume that $\lim_{n\to\infty}\int {\tau} d\mu_n=\lambda_2\int{\tau} d\mu$.  Then $(\widehat{\mu}_n)_n$ converges in the cylinder topology of $\widehat{\Sigma}$ to $\frac{\lambda_1}{\lambda_2}\widehat{\mu}$. 
\end{lemma}
\begin{proof} 
Let  $C\subseteq \widehat{\Sigma}$ be a cylinder. By Remark \ref{rem:cyl}, there exists a cylinder $D\subseteq \Sigma$ such that $\widehat{\mu}(C)=\frac{1}{\int {\tau} d\mu}\mu(D),$ for every $\mu\in \M(\sigma;\tau)$. Then
$$\lim_{n\to\infty}\widehat{\mu}_n(C)=\lim_{n\to\infty}\frac{1}{\int {\tau} d\mu_n}\mu_n(D)=\frac{\lambda_1}{\lambda_2}\frac{1}{\int{\tau} d\mu}\mu(D)=\frac{\lambda_1}{\lambda_2}\widehat{\mu}(C).$$
We conclude that $(\widehat{\mu}_n)_n$ converges on cylinders to $\frac{\lambda_1}{\lambda_2}\widehat{\mu}$. 
\end{proof}

\subsection{Entropy at infinity of $(\widehat{\Sigma},\widehat{\sigma})$} Abramov's formula states that $h_{\widehat{\mu}}(\widehat{\sigma})=\frac{1}{\int {\tau} d\mu}h_\mu(\sigma),$ for every $\mu\in \M(\sigma;\tau)$ (see Remark \ref{rem:discretesus}). It follows that $h_{top}(\widehat{\sigma})=\inf\{t\in\R: P(-t{\tau})\le 0\}.$

\begin{remark}\label{rem:finent} Since $\tau\ge 1$, the map $t\mapsto P(-t\tau)$ is strictly decreasing. In particular, if $P(-{\tau})<\infty$, then $h_{top}(\widehat{\sigma})<\infty$. \end{remark}

\begin{proposition}\label{prop:Gs} The following holds:
\begin{enumerate} 
\item\label{cas1} If $\tau\in \Psi_1$, then $h_\infty(\widehat{\sigma})=s_{\tau}$. 
\item\label{cas2} If $\tau \in \Psi_2$, then $h_\infty(\widehat{\sigma})= s_\infty(-\tau)$. 
\end{enumerate}
In other words, if $\tau\in\Psi$, then $h_\infty(\widehat{\sigma})=h(\tau)$ (see Remark \ref{rem:inutil}).
\end{proposition}
\begin{proof}

Case (\ref{cas1}): There exists a sequence $(\mu_n)_n$ in $\M(\sigma;\tau)$ which converges to the zero measure and such that $\lim_{n\to\infty}(h_{\mu_n}(\sigma)- s_\tau\int \tau d\mu_n)=0$. Then, by Abramov's formula $\lim_{n\to\infty}h_{\widehat{\mu}_n}=s_\tau,$  and $(\widehat{\mu}_n)_n$ converges to the zero measure (see Lemma \ref{tozero}). It follows that $h_\infty(\widehat{\sigma})\ge{s}_{\tau}$. 

Let $(\nu_n)_n$ a sequence of measures in $\M(\widehat{\sigma})$ that converges to the zero measure, and such that $\lim_{n\to\infty}h_{\nu_n}(\widehat{\sigma})=h_\infty(\widehat{\sigma})$ (the existence of such a sequence is given by Theorem \ref{compact}).
 Let $\eta_n$ be the measure in   $\M(\sigma;\tau)$ such that $\widehat{\eta}_n =\nu_n$. Then, 
$$\lim_{n\to\infty}\bigg(\frac{h_{\eta_n}(\sigma)-h_\infty(\widehat{\sigma})\int \tau d\eta_n}{\int \tau d\eta_n}\bigg)=0.$$
If $\limsup_{n\to\infty}\int \tau d\eta_n<\infty$, then $(\eta_n)_n$ converges on cylinders to the zero measure (see Lemma \ref{tozero}) and  $\lim_{n\to\infty}\big(h_{\eta_n}(\sigma)-h_\infty(\widehat{\sigma})\int \tau d\eta_n\big)=0.$
In particular, ${P}_\infty(-h_\infty(\widehat{\sigma})\tau)\ge 0$, and therefore ${s}_{\tau}\ge h_\infty(\widehat{\sigma})$. If $\limsup_{n\to\infty}\int \tau d\eta_n=\infty$, then there exists a subsequence for which the integral diverges. Maybe after passing to a subsequence we can assume that $\lim_{n\to\infty}\int \tau d\eta_n=\infty$. Fix $\e>0,$ and note that 
\begin{align}\label{s}  \limsup_{n\to\infty} \frac{h_{\eta_n}(\sigma)}{\int \tau d\eta_n}\le (s_\infty(-\tau)+\e)+\limsup_{n\to\infty}\frac{P(-(s_\infty(-\tau)+\e)\tau)}{\int \tau d\eta_n}=s_\infty(-\tau)+\e.\end{align}
Therefore, $h_\infty(\widehat{\sigma})\le s_\infty(-\tau)$. This is not possible since we already proved that  $h_\infty(\widehat{\sigma})\ge{s}_{\tau}$.\\

Case (\ref{cas2}): Set $s=s_\infty(-\tau)$. We claim that $h_\infty(\widehat{\sigma})\ge s$. We assume that $s>0$, otherwise there is nothing to prove. Fix $\epsilon>0$ small. Since $P(-(s-\epsilon)\tau)=\infty$, there exists a sequence  $(\mu_n)_n$ in  $\M(\sigma;\tau)$ such that 
\begin{align}\label{eq:pp}\lim_{n\to\infty}\bigg(h_{\mu_n}(\sigma)-(s-\epsilon)\int \tau d\mu_n\bigg)=\infty.\end{align} In particular, the limit is greater than zero and therefore $\liminf_{n\to\infty}h_{\widehat{\mu}_n}(\widehat{\sigma})\ge s-\epsilon$. It is enough to prove that $\limsup_{n\to\infty}\int \tau d\mu_n=\infty$, in which case $(\widehat{\mu}_n)_n$ has a subsequence that converges to the zero measure (see Lemma \ref{tozero}) and $\epsilon>0$ was arbitrary. Suppose otherwise that $\limsup_{n\to\infty}\int \tau d\mu_n<\infty$, and therefore there exists $M$ such that $\int \tau d\mu_n\le M$, for all $n\in \N$. It follows that $h_{\mu_n}(\sigma)\le M+P(-\tau)$, which contradicts equation (\ref{eq:pp}).

We now claim that $h_\infty(\widehat{\sigma})\le s$. Let $(\nu_n)_n$ a sequence in $\M(\widehat{\sigma})$ which converges to the zero measure and such that $\lim_{n\to\infty}h_{\nu_n}(\widehat{\sigma})=h_\infty(\widehat{\sigma})$. Let $\eta_n$ be the measure in   $\M(\sigma;\tau)$ such that $\widehat{\eta}_n =\nu_n$. Then, 
$$\lim_{n\to\infty}\bigg(\frac{h_{\eta_n}(\sigma)-h_\infty(\widehat{\sigma})\int \tau d\eta_n}{\int \tau d\eta_n}\bigg)=0.$$
If $\limsup_{n\to\infty}\int \tau d\eta_n<\infty$, then $\lim_{n\to\infty}\big(h_{\eta_n}(\sigma)-h_\infty(\widehat{\sigma})\int \tau d\eta_n\big)=0,$ and $(\eta_n)_n$ converges to the zero measure (see Lemma \ref{tozero}).  In particular, $P_\infty(-h_\infty(\widehat{\sigma})\tau)\ge 0$. If $h_\infty(\widehat{\sigma})> s$, then $\tau\in \Psi_1$, which contradicts our assumption. We therefore have that $h_\infty(\widehat{\sigma})\le s$. Now suppose that $\limsup_{n\to\infty}\int \tau d\nu_n=\infty.$ Maybe after passing to a subsequence we can assume that $\lim_{n\to\infty}\int \tau d\nu_n=\infty$. Inequality (\ref{s}) implies that $h_\infty(\widehat{\sigma})\le s$. 
\end{proof}

\subsection{Proof of Theorem \ref{thm:bsctm}} In this subsection, we will prove Theorem \ref{thm:bsctm}, a compactness result for sequences of invariant probability measures on countable Markov shifts that may not have the $\F$-property. This result will be frequently used in later applications concerning the existence of equilibrium states and maximizing measures.

\begin{proof}[Proof of Theorem \ref{thm:bsctm}] Consider $M \in \mathbb{R}$ such that $\sup \phi < M$. Define $\psi = -(\phi - M)$. Note that $\psi \in \Psi$. Define $\tau:\Sigma\to\R,$ by $\tau(x) = \sup_{y \in [x_1, x_2]} \lceil \psi(y) \rceil$, where $x = (x_1, x_2, \ldots)\in\Sigma$. Since $\text{var}_2(\phi)=\text{var}_2(\psi)$ is finite, it follows that $\|\tau - \psi\|_0 < \infty$. In particular, $P(-\tau)<\infty$. Moreover, $\tau$ takes values in $\mathbb{N}$ and $\text{var}_2(\tau) = 0$. 

Let $(\widehat{\Sigma}, \widehat{\sigma})$ be the CMS constructed at the beginning of this section for the pair $(\Sigma, \sigma)$ and $\tau$. Note that $(\widehat{\Sigma},\widehat{\sigma})$ is transitive, because $(\Sigma,\sigma)$ is transitive. Additionally, $h_{top}(\widehat{\sigma})$ is finite (see Remark \ref{rem:finent}). By Theorem \ref{compact}, $\M_{\le1}(\widehat{\sigma})$  is compact with respect to the cylinder topology. As before, we denote by $\widehat{\nu}$ the measure in $\M(\widehat{\sigma})$ associated to $\nu\in \M(\sigma;\tau)$. Maybe after passing to a subsequence, we can assume that $(\widehat{\mu}_n)_n$ converges on cylinders to $\lambda \widehat{\mu}$, where $\lambda\in [0,1]$ and $\mu\in \M(\sigma;\tau)$. Maybe after passing to a further subsequence we can assume that $A:=\lim_{n\to\infty}\int \tau d\mu_n$ is well defined and finite (the limsup is finite by assumption).  Let $C$ be a cylinder in $\Sigma$. By Remark \ref{rem:cyl}, there exists a cylinder $D\subseteq \widehat{\Sigma}$ such that $\widehat{\nu}(D)=\frac{1}{\int \tau d\nu}\nu(C)$, for every $\nu\in\M(\sigma;\tau)$. Then,
$\lim_{n\to\infty}\mu_n(C)= \frac{\lambda A}{\int \tau d\mu}\mu(C).$ We conclude that $(\mu_n)_n$ converges on cylinders to $\mu_0=\frac{\lambda A}{\int \tau d\mu}\mu$.  Note that $\int \tau d{\mu_0}<\infty$, and therefore $\int \phi d{\mu}_0>-\infty$ (see Lemma \ref{tired}). \end{proof}

The next result follows from Theorem \ref{thm:bsctm} and will be used in the proof of Lemma \ref{lem:mvstau}.

\begin{corollary}\label{lem:ui} Let $\tau\in \Psi$ and $m\in \N$. Let $(\mu_n)_n$ be a sequence in $\M(\sigma),$ and $\mu\in\M_{\le1}(\sigma)$. Assume that $\limsup_{n\to\infty}\int \tau d\mu_n<\infty$. Then, $(\mu_n)_n$ converges on cylinders to $\mu$ if and only if $\lim_{n\to\infty}\mu_n(C)=\mu(C)$, for every cylinder $C\subseteq\Sigma$ of length $m$.
\end{corollary}
\begin{proof} Suppose that $\lim_{n\to\infty}\mu_n(C)=\mu(C)$, for every cylinder $C\subseteq\Sigma$ of length $m$. It follows by Theorem \ref{thm:bsctm} (applied to the potential $-\tau$) that every subsequence of  $(\mu_n)_n$ has a subsubsequence which converges on cylinders to a countably additive measure $\nu\in\M_{\le 1}(\sigma)$. Since $\nu(C)=\mu(C)$, for every cylinder of length $m$, it follows that $\nu=\mu$. We conclude that $(\mu_n)_n$ converges on cylinders to $\mu$. The forward implication follows by definition. 
\end{proof}

\section{Upper semi-continuity properties of the pressure map}\label{infinite}

In this section, we establish some upper semi-continuity results for the pressure map of uniformly continuous potentials. We will prove Theorem \ref{thm:1}, which generalizes Theorem \ref{itv} and allows for the consideration of countable Markov shifts with infinite entropy and mild regularity assumptions on the potential.
 
\subsection{A first upper semicontinuity result for the pressure}\label{subsec:51}


\begin{proposition}\label{theo:3}  Let $(\Sigma,\sigma)$ be a transitive CMS and $\tau:\Sigma\to\N$ a potential such that  $P(-\tau)$ is finite and $\emph{var}_2(\tau)=0$. Let $(\mu_n)_n$ be a sequence  in $\M_{-\tau}(\sigma)$ which converges on cylinders to $\lambda\mu$, where $\lambda\in [0,1]$ and $\mu\in \M_{-\tau}(\sigma)$. Assume that $\limsup_{n\to\infty}\int \tau d\mu_n<\infty$. Then,
\begin{align}\label{ineq:p}\limsup_{n\to\infty}\big(h_{\mu_n}(\sigma)-h(\tau)\int \tau d\mu_n\big)\le \lambda\big(h_{\mu}(\sigma)-h(\tau)\int \tau d\mu\big).\end{align}
\end{proposition}
\begin{proof} It is sufficient to prove the proposition in the case where $(\int \tau d\mu_n)_n$ is convergent (otherwise we can consider a subsequence for which the limsup in  (\ref{ineq:p}) is a limit, and then pass to a further subsequence for which the integral converges). By Lemma \ref{tired}, we have that $\int \tau d\mu<\infty$. Set $\lambda_0$ such that $\lim_{n\to\infty}\int \tau d\mu_n=\lambda_0\int \tau d(\lambda \mu)$, which, by Lemma \ref{tired}, also  satisfies that $\lambda_0\ge 1$.

Let $(\widehat{\Sigma},\widehat{\sigma})$ be the CMS associated with $(\Sigma,\sigma)$ and $\tau$  constructed in Section \ref{con}. By Remark \ref{rem:finent}, $h_{top}(\widehat{\sigma})$ is finite. It follows by Lemma \ref{topology} that $(\hat{\mu}_n)_n$ converges  on cylinders to $\frac{1}{\lambda_0}\hat{\mu}$. We use Theorem \ref{itv} and Proposition \ref{prop:Gs} to conclude that
$$\limsup_{n\to\infty}\frac{h_{\mu_n}(\sigma)}{\int \tau d\mu_n}\le \frac{1}{\lambda_0}\frac{h_\mu(\sigma)}{\int \tau d\mu}+\bigg(1-\frac{1}{\lambda_0}\bigg)h(\tau).$$
Then,
$$\limsup_{n\to\infty} h_{\mu_n}(\sigma) \le \lambda h_\mu(\sigma) +\bigg(\lim_{n\to\infty}\int \tau d\mu_n-\lambda \int \tau d\mu\bigg)h(\tau),$$
which is equivalent to inequality (\ref{ineq:p}).
\end{proof}

We will extend Proposition \ref{theo:3} in order to include potentials in $\Psi$. Let $\Psi_n$ be  the class of potentials in $\Psi$ which are locally constant and depend only on the first $n$ coordinates of $\Sigma$; that is, their $n$-variation is zero. First, we will prove that Proposition \ref{theo:3} holds for potentials in $\Psi_2$, removing the assuming that the potential takes values in $\N$. Next, we will prove that the proposition holds for potentials in $\Psi_n$, and finally, that it holds for potentials in $\Psi$. We begin with a lemma.

\begin{lemma}\label{lem:qk} Let $(\phi_n)_n, \phi$ be potentials in $\Psi$  such that $\lim_{k\to\infty}\|\phi-\phi_k\|_0=0.$ Then $\lim_{n\to\infty}h(\phi_n)=h(\phi)$.\end{lemma}

\begin{proof} If $\|\phi-\phi_k\|<\infty$, then $P(-t\phi)=\infty$ if and only if $P(-t\phi_k)=\infty$. It follows that $s_\infty(-\phi)=s_\infty(-\phi_k)$. 
By Lemma \ref{convex}, we have also that $\lim_{k\to\infty} P_\infty(-t\phi_k)=P_\infty(-t\phi)$.

Let us first consider the case where $\phi\in\Psi_1$. Choose $\epsilon>0$ small such that $P_\infty(-(s_\phi-\epsilon)\phi)>0>P_\infty(-(s_\phi+\epsilon)\phi)$. Then $P_\infty(-(s_\phi-\epsilon)\phi_k)>0>P_\infty(-(s_\phi+\epsilon)\phi_k)$, for large enough $k$, and therefore $|s_\phi-s_{\phi_k}|<\epsilon$. It follows that for large enough $k$ we have that $\phi_k\in\Psi_1$, and that $\lim_{k\to\infty}s_{\phi_k}=s_\phi$, which proves the claim in this case.

Suppose that $\phi\in\Psi_2$. For every $\epsilon>0$ we have that $P_\infty(-(s_\infty(-\phi)+\epsilon)\phi)<0$.  Therefore, for large enough $k$ we have that $P_\infty(-(s_\infty(-\phi)+\epsilon)\phi_k)<0$, and thus $s_{\phi_k}\le s_\infty(-\phi)+\epsilon$. This implies that $\limsup_{n\to\infty} s_{\phi_k}\le s_\infty(-\phi)$. Since $s_\infty(-\phi)=s_\infty(-\phi_k)$, and $h(\phi)=s_\infty(-\phi)$, the claim follows.
\end{proof}

\begin{proposition}\label{lem:lc2} Proposition \ref{theo:3} holds when $\tau\in\Psi_2$.
\end{proposition}
\begin{proof} Define $\tau_k(x)=\frac{1}{k}\lceil  k\tau(x)\rceil,$ and observe that $\lim_{k\to 0}\|\tau-\tau_k\|_0=0$. Set $f_k=k\tau_k,$ and note that $h(f_k)=\frac{1}{k}h(\tau_k)$. The function $f_k$ takes values in $\N$ and $\text{var}_2(f_k)=0$. Proposition \ref{theo:3} holds for the functions $f_k$, which is equivalent to say that the proposition holds for $\tau_k$. By Lemma \ref{lem:qk}, it follows that $\lim_{k\to\infty}h(\tau_k)=h(\tau)$. Additionally, $|\int \tau_k d\nu-\int \tau d\nu|\le \|\tau-\tau_k\|_0$, for every $\nu\in\M(\sigma)$. All of this together implies that Proposition \ref{theo:3} holds for $\tau\in\Psi_2$.
\end{proof}

In subsection \ref{recode} we constructed a CMS $(\Sigma_m,\sigma_m)$, which is a recoding of $(\Sigma,\sigma)$, and defined a canonical conjugacy $\pi_m:\Sigma_m\to\Sigma$. The map $(\pi_m)_*:\M_{\le1}(\sigma_m)\to\M_{\le1}(\sigma)$ is a bijection. For $\mu\in \M_{\le1}(\sigma)$ we defined  $\tilde{\mu}=(\pi_m)^{-1}_*{\mu}\in \M_{\le1}(\sigma_m)$. 

\begin{remark}\label{rem:preclaim} Let $(\mu_n)_n$ be a sequence in $\M(\sigma)$ such that $\limsup_{n\to\infty}\int \tau d\mu_n<\infty$, where $\tau\in \Psi$. Then, $(\mu_n)_n$ converges on cylinders to $\mu\in \M_{\le1}(\sigma)$ if and only if $(\tilde{\mu}_n)_n$ converges on cylinders to $\tilde{\mu}$ (see Remark \ref{rem:topm} and Corollary \ref{lem:ui}).
\end{remark}

In the next lemma we have added superscripts $\Sigma$ and $\Sigma_m$ to the pressure to indicate the CMS we are considering. Let $\tau_m=\tau\circ \pi_m$. 
\begin{lemma}\label{lem:mvstau} $h(\tau)=h(\tau_m)$.\end{lemma}
\begin{proof} Observe that  $h_\mu(\sigma)=h_{\tilde{\mu}}(\sigma_m)$, and $\int \tau d\mu=\int \tau_m d\tilde{\mu}$, for every $\mu\in \M(\sigma)$. It follows that $P^{\Sigma_m}(-t\tau_m)=P^{\Sigma}(-t\tau)$, and therefore $s_\infty(-\tau_m)=s_\infty(-\tau)$. 

We will prove that if $t>s_\infty(-\tau)$, then  $P_\infty^{\Sigma_m}(-t\tau_m)=P^{\Sigma}_\infty(-t\tau)$. It follows from this fact that $\tau\in \Psi_1$ if and only if $\tau_m\in \Psi_1$, and that ${s}_\tau={s}_{\tau_m}$, which finishes the proof of the lemma, as we can conclude that $h(\tau)=h(\tau_m)$.

If there exists a sequence $(\mu_n)_n$ in $\M(\sigma)$ that converges on cylinders to the zero measure, and such that $\lim_{n\to\infty}\big(h_{\mu_n}(\sigma)-t\int \tau d\mu_n\big)=P_\infty(-t\tau)$, then $(\tilde{\mu}_n)_n$ converges to the zero measure (see Remark \ref{rem:topm}), and $\lim_{n\to\infty}(h_{\tilde{\mu}_n}(\sigma_m)-t\int \tau_m d\tilde{\mu}_n)=P_\infty(-t\tau)$. Therefore, we have $P_\infty^{\Sigma_m}(-t\tau_m)\ge P^{\Sigma}_\infty(-t\tau)$. If no such sequence exists, then we still have the inequality $P_\infty^{\Sigma_m}(-t\tau_m)\ge P^{\Sigma}_\infty(-t\tau)$. Note that if  $P_\infty^{\Sigma_m}(-t\tau_m)=-\infty$, then $P_\infty^{\Sigma}(-t\tau)=-\infty$. We will now suppose that $P_\infty^{\Sigma_m}(-t\tau_m)>-\infty$. 

Let  $(\tilde{\mu}_n)_n$ be a sequence in $\M_{-\tau_m}(\sigma_n)$  that converges on cylinders to the zero measure and such that $P_\infty^{\Sigma_m}(-t\tau_m)=\lim_{n\to\infty}(h_{\tilde{\mu}_n}(\sigma_m)-t\int \tau_m d\tilde{\mu}_n)$. Since $t>s_\infty(-\tau_m),$ we have that $\limsup_{n\to\infty}\int \tau_m d\tilde{\mu}_n<\infty$ (see  Lemma \ref{prop:gapr}); equivalently, $\limsup_{n\to\infty}\int \tau d\mu_n<\infty$. By Remark \ref{rem:preclaim}, $(\mu_n)_n$ converges to the zero measure, and therefore $P_\infty^{\Sigma_m}(-t\tau_m)\le P^{\Sigma}_\infty(-t\tau)$. We conclude that $P_\infty^{\Sigma_m}(-t\tau_m)=P^{\Sigma}_\infty(-t\tau)$, for every $t>s_\infty(-\tau)$.
\end{proof}

\begin{proposition}\label{toF_n2}   Proposition \ref{theo:3} holds when $\tau\in\Psi_m$, for $m\ge3$.
\end{proposition}
\begin{proof} 
Let $(\mu_n)_n$ be a sequence in $\M_{-\tau}(\sigma)$ which converges on cylinders to $\lambda\mu$, where $\mu\in \M_{-\tau}(\sigma),$ and $\lambda\in [0,1]$. By assumption, $\limsup_{n\to\infty}\int \tau d\mu_n<\infty$. Consider $(\Sigma_m,\sigma_m)$, and note that $\tau_m\in\Psi_1$, as it only depends on the first coordinate of $\Sigma_m$. We can apply Proposition \ref{lem:lc2} to the pair $(\Sigma_m,\sigma_m)$ and $\tau_m$. Since $(\tilde{\mu}_n)_n$ converges on cylinders to $\lambda\tilde{\mu}$ (see Remark \ref{rem:topm}), we have that
\begin{align*}\limsup_{n\to\infty}\big(h_{\tilde{\mu}_n}(\sigma_m)-h(\tau_m)\int \tau_m d{\tilde{\mu}_n}\big)\le \lambda\big(h_{\tilde{\mu}}(\sigma_m)-h(\tau_m)\int \tau_m d{\tilde{\mu}}\big).\end{align*}
Finally, recall that  $h_\nu(\sigma)=h_{\tilde{\nu}}(\sigma_m)$, and $\int \tau d\nu=\int \tau_m d\tilde{\nu}$, for every $\nu\in \M(\sigma)$, and that $h(\tau)=h(\tau_m)$ (see Lemma \ref{lem:mvstau}). We conclude that inequality (\ref{ineq:p}) holds for $\tau$. 
\end{proof}

\begin{theorem}\label{theo:4}  Let $(\Sigma,\sigma)$ be a transitive CMS and $\tau\in\Psi$. Let $(\mu_n)_n$ be a sequence  in $\M_{-\tau}(\sigma)$ which converges on cylinders to $\lambda\mu$, where $\mu\in \M_{-\tau}(\sigma)$, and $\lambda\in [0,1]$. Suppose that $\limsup_{n\to\infty}\int \tau d\mu_n<\infty$. Then,
$$\limsup_{n\to\infty}\big(h_{\mu_n}(\sigma)-h(\tau)\int \tau  d\mu_n\big)\le \lambda\big(h_{\mu}(\sigma)-h(\tau)\int \tau d\mu\big).$$
\end{theorem}
\begin{proof} Since $\tau\in \Psi$ we have that $\text{var}_2(\tau)<\infty$, and that $\lim_{n\to\infty}\text{var}_n(\tau)=0$. Fix $m\ge 2$ and define $\tau_m(x)=\sup\{\tau(y):y\in[x_1,\ldots, x_{m}]\}$, where $x=(x_1,x_2,\ldots)$. Note that $\|\tau-\tau_m\|_0\le \text{var}_m(\tau)$, and that $\text{var}_2(\tau_m)<\infty$. Moreover, $\tau_m\in\Psi_m$ and by Proposition \ref{toF_n2}  we have that 
$$\limsup_{n\to\infty}\big(h_{\mu_n}(\sigma)-h(\tau_m)\int \tau_m d\mu_n\big)\le \lambda\big(h_{\mu}(\sigma)-h(\tau_m)\int \tau_m d\mu\big).$$
Since $|\int \tau_m d\nu-\int \tau d\nu|\le \|\tau_m-\tau\|_0$, for every $\nu\in \M(\sigma)$, and $\|\tau_m-\tau\|_0\to 0$, we obtain the desired inequality (see Lemma \ref{lem:qk}).
\end{proof}

\subsection{Proof of Theorem \ref{thm:1}} This subsection is devoted to the proof of Theorem \ref{thm:1}. First, we prove that the pressure map is upper semi-continuous on the full shift over a countable alphabet for a suitable class of potentials. For a general CMS our strategy is to induce on a cylinder of length one and use the results obtained for the full shift.

\begin{lemma}\label{lem:vale1} Let $(\N^\N,\sigma)$ be the full shift. Let $\phi\in C_{uc}(\N^\N)$ be a potential such that $\emph{var}_1(\phi)$ and $P(-\phi)$ are finite. Then, \begin{align}\label{valep}\lim_{n\to\infty} \inf_{x\in [n]} \phi(x)=\infty.\end{align}
Moreover, if a sequence $(\nu_n)_n$ in $\M(\sigma)$ converges on cylinders to the zero measure, then $\lim_{n\to\infty}\int \phi d\nu_n=\infty$. 
\end{lemma}
\begin{proof} Let $a:\N\to \R,$ be given by $a(n)=\inf_{x\in [n]}\phi(x)$. Define $\phi_1:\N^\N\to \R$, where   $\phi_1(x)=a(k),$ if $x\in [k]$. Note that $\|\phi-\phi_1\|_0\le \text{var}_1(\phi)$. In particular, $P(-\phi_1)=\log\big(\sum_{n=1}^\infty e^{-a(n)}\big)$ is finite (see Example \ref{ex:full}). We conclude that $\lim_{n\to\infty} a(n)=\infty$.

Let us now suppose that  $(\nu_n)_n$ converges on cylinders to the zero measure. We argue by contradiction and suppose that $\liminf_{n\to\infty}\int \phi d\nu_n<\infty$. In this case, the inequality
$$1-\nu_n(\bigcup_{k< n}[k])=\nu_n(\bigcup_{k\ge n}[k])\le \frac{\int \phi d\nu_n}{\inf \{\phi(x):x\in \bigcup_{k\ge n}[k]\}}=\frac{\int \phi d\nu_n}{\inf_{k\ge n} a(k)}, $$
implies that $(\nu_n)_n$ does not converge on cylinders to the zero measure.\end{proof}

\begin{remark}\label{rem:fshift} If $\phi: \N^\N \to \R$ is a potential such that both $P(-\phi)$ and $\text{var}_1(\phi)$ are finite, then it follows by Lemma \ref{lem:vale1} that $\phi$ is bounded below. Additionally, suppose $\phi \in \Psi$ and $t > s_\infty(-\phi)$. In this case, $P_\infty(-t\phi) = -\infty$ (see Lemma \ref{prop:gapr}
 and Lemma \ref{lem:vale1}). In particular, $\phi \in \Psi_2$, and therefore $h(\phi)=s_\infty(-\phi)$.
\end{remark}

\begin{proposition}\label{cor:v} Let $(\N^\N,\sigma)$ be the full shift and $\phi\in C_{uc}(\N^\N)$ a potential such that $\emph{var}_1(\phi)$ and $P(-\phi)$ are finite. Let $(\mu_n)_n$ be a sequence in $\M_{-\phi}(\sigma)$ which converges weak$^*$ to $\mu\in \M_{-\phi}(\sigma)$. Suppose that $\limsup_{n\to\infty} \int \phi d\mu_n<\infty$. Then, $$\limsup_{n\to\infty}\big( h_{\mu_n}(\sigma)-\int \phi d\mu_n\big)\le h_\mu(\sigma)-\int \phi d\mu.$$
\end{proposition}
\begin{proof} By Remark \ref{rem:fshift}, the potential $\phi$ is bounded below. After adding a positive constant, we can assume that $\phi$ is bounded away from zero, and therefore $\phi\in \Psi$. By Remark \ref{rem:fshift}, we also have that $\phi\in \Psi_2$, and therefore $h(\phi)=s_\infty(-\phi)$. It follows by Theorem \ref{theo:4} that
$$\limsup_{n\to\infty}\big( h_{\mu_n}(\sigma)-s_\infty(-\phi)\int \phi d\mu_n\big)\le h_\mu(\sigma)-s_\infty(-\phi)\int \phi d\mu.$$
Lemma \ref{tired} implies that 
$$\limsup_{n\to\infty}-(1-s_\infty(-\phi))\int \phi d\mu_n\le -(1-s_\infty(-\phi))\int \phi d\mu.$$
The result follows from adding these inequalities together.
\end{proof}
 
\begin{remark} Under the assumptions in Proposition \ref{cor:v}, assume further that $s_\infty(-\phi)<1$. If  $\lim_{n\to\infty}\big(h_{\mu_n}(\sigma)-\int \phi d\mu_n\big)=h_\mu(\sigma)-\int \phi d\mu$,  then $\lim_{n\to\infty}\int \phi d\mu_n=\int \phi d\mu$, and $\lim_{n\to\infty}h_{\mu_n}(\sigma)=h_\mu(\sigma)$.
\end{remark}

The next proposition is the key ingredient in the proof of Theorem \ref{thm:1}.

\begin{proposition}\label{prop:5}  Let $(\Sigma,\sigma)$ be a mixing CMS and $\phi$ a weakly H\"older potential such that $P(-\phi)$ and $\emph{var}_2(\phi)$ are finite.  Let $(\mu_n)_n$ be a sequence  in $\M_{-\phi}(\sigma)$ which converges on cylinders to $\lambda\mu$, where $\lambda\in [0,1]$ and $\mu\in \M_{-\phi}(\sigma)$. Suppose that $\limsup_{n\to\infty}\int \phi  d\mu_n<\infty$. Then,
$$\limsup_{n\to\infty}\big(h_{\mu_n}(\sigma)-\int \phi d\mu_n\big)\le \lambda\big(h_\mu(\sigma)-\int \phi d\mu\big)+(1-\lambda)P(-\phi).$$
\end{proposition}

\begin{proof} 

The case $\lambda=0$ follows from the variational principle. Suppose that  $\lambda>0$. We will further suppose that  $P(-\phi)=0$ (otherwise, we can substract $P(-\phi)$ from both sides and consider the potential $\phi'=\phi+P(-\phi)$).

 Fix  $a\in S$ such that $\mu([a])>0$. We will induce on the 1-cylinder $[a]$. As described in subsection \ref{sec:spr}, the induced system $(\Sigma(a),\sigma_a)$ is topologically conjugate to the full shift, and cylinders in $\Sigma(a)$ are certain cylinders in $\Sigma$. Let $\tau:\Sigma(a)\to\N$ be the first return time map, and  $\overline{\phi}:\Sigma(a)\to\R$ the induced potential, which is defined by $\overline{\phi}(x)=\sum_{i=0}^{\tau(x)-1}\phi(\sigma^i x)$. Let $\M^{[a]}(\sigma)=\{\mu\in \M(\sigma): \mu([a])>0\}$. For $\nu\in\M^{[a]}(\sigma)$, define $\overline{\nu}(\cdot)=\frac{1}{\nu([a])}\nu(\cdot\cap [a]),$ which is a measure supported on $\Sigma(a)$. By Kac's formula we have that $\int \tau d\overline{\nu}=\frac{1}{\nu([a])}$, and $\int \phi d\nu=\frac{1}{\int \tau d\overline{\nu}}\int \overline{\phi} d\overline{\nu}$, for every $\nu\in \M^{[a]}(\sigma)$. Moreover, the  map $\nu\mapsto \overline{\nu}$ is a bijection between $\M^{[a]}(\sigma)$ and $\M_\tau(\sigma_a)$.

 Since $\lim_{n\to\infty}\mu_n([a])=\lambda\mu([a])$, we can assume that $\mu_n([a])>0$, for all $n\in\N$. Then $\mu, (\mu_n)_n$ are in $\M^{[a]}(\sigma)$. Since $(\mu_n)_n$ converges on cylinders to $\lambda\mu$, we have that \begin{align}\label{va} \lim_{n\to\infty}\int \tau d\overline{\mu}_n=\lim_{n\to\infty}\frac{1}{\mu_n([a])}=\frac{1}{\lambda\mu([a])}=\frac{1}{\lambda}\int \tau d\overline{\mu}.\end{align}
Moreover, since cylinders in $\Sigma(a)$ are cylinders in $\Sigma$, it follows that $$\lim_{n\to\infty}\overline{\mu}_n(C)=\lim_{n\to\infty}\frac{1}{\mu_n([a])}\mu_n(C)=\frac{1}{\mu([a])}\mu(C)=\overline{\mu}(C),$$
for every cylinder $C\subseteq \Sigma(a)$. In particular, $(\overline{\mu}_n)_n$ converges weak$^*$ to $\overline{\mu}$ (see Remark \ref{rem:compatibility}). Set $M:=\sup_n \int \phi d\mu_n$. It follows from Kac's formula that $\int \overline{\phi}d\overline{\mu}_n \le M  \int \tau d\overline{\mu}_n$, and therefore $\limsup_{n\to\infty} \int \overline{\phi}d\overline{\mu}_n<\infty$.  Note that $\text{var}_1(\overline{\phi})\le \text{var}_2(\phi)<\infty$. 

We will consider two cases:

{\bf Case 1} ($-\phi$ is recurrent): In this case we have that $P(-\overline{\phi})=0$ (see Remark \ref{rem:induce0}). It follows by Proposition \ref{cor:v} that 
$$\limsup_{n\to\infty}\big( h_{\overline{\mu}_n}(\sigma_a)-\int \overline{\phi} d\overline{\mu}_n\big)\le h_{\overline{\mu}}(\sigma_a)-\int \overline{\phi} d\overline{\mu}.$$
Combining this with (\ref{va}), we obtain that 
\begin{align*} \limsup_{n\to\infty}\big( h_{\mu_n}(\sigma)-\int \phi d\mu_n\big)=\limsup_{n\to\infty}\bigg(\frac{h_{\overline{\mu}_n}(\sigma_a)-\int \overline{\phi} d\overline{\mu}_n}{\int \tau d\overline{\mu}_n}\bigg)&\le \lambda\bigg(\frac{h_{\overline{\mu}}(\sigma_a)-\int \overline{\phi} d\overline{\mu}}{\int \tau d\overline{\mu}}\bigg)\\ 
 &=\lambda\big(h_\mu(\sigma)-\int \phi d\mu\big).\end{align*}

{\bf Case 2} ($-\phi$ is transient): Set $t_0 := \Delta_a(-\phi) < 0$, where $\Delta_a(-\phi)$ is the $a$-discriminant of $-\phi$ (see equation (\ref{discriminant})). By the discriminant theorem (see Theorem \ref{disc}) $-\phi + t1_{[a]}$ is transient for every $t \in [0, -t_0)$. It follows that $P(-\phi + t1_{[a]}) = 0,$ for every $t \in [0, -t_0]$ (see \cite[Lemma 4.1]{cs}). Since $\Delta_a(-\phi - t_0 1_{[a]}) = 0$, the discriminant theorem implies that $-\phi_0 := -\phi - t_0 1_{[a]}$ is recurrent. Note that $P(-\phi_0)=0$, and that  $-\phi_0$ satisfies the assumption in Case 1, so we obtain that 
$$\limsup_{n\to\infty}\big(h_{\mu_n}(\sigma)-\int \phi d\mu_n-t_0\mu_n([a]) \big)\le \lambda\big(h_\mu(\sigma)-\int \phi d\mu-t_0\mu([a])\big).$$
Since $\lim_{n\to\infty}\mu_n([a])=\lambda\mu([a])$, we conclude that 
\begin{align}\label{ine:xxx}\limsup_{n\to\infty}\big(h_{\mu_n}(\sigma)-\int \phi d\mu_n \big)\le \lambda\big(h_\mu(\sigma)-\int \phi d\mu\big).\end{align}
In particular, inequality (\ref{ine:xxx}) holds in both cases.
\end{proof}

\begin{proposition}\label{theo:valecul}  Let $(\Sigma,\sigma)$ be a transitive CMS and $\phi\in C_{uc}(\Sigma)$ a potential such that $P(-\phi)$ and $\emph{var}_2(\phi)$ are finite.  Let $(\mu_n)_n$ be a sequence  in $\M_{-\phi}(\sigma)$ which converges on cylinders to $\lambda\mu$, where $\lambda\in [0,1]$ and $\mu\in\M_{-\phi}(\sigma)$. Suppose that $\limsup_{n\to\infty}\int \phi  d\mu_n<\infty$. Then,
$$\limsup_{n\to\infty}\big(h_{\mu_n}(\sigma)-\int \phi d\mu_n\big)\le \lambda\big(h_\mu(\sigma)-\int \phi d\mu\big)+(1-\lambda)P_\infty(-\phi).$$
\end{proposition}
\begin{proof} We will assume that $(\Sigma,\sigma)$ is mixing and explain how to derive the transitive case from the mixing case in the remark below. Fix $m\ge 2$ and define $\phi_m(x)=\sup\{\phi(y):y\in[x_1,\ldots, x_{m}]\}$, where $x=(x_1,x_2,\ldots)$. Note that $\|\phi-\phi_m\|_0\le \text{var}_m(\phi)$, and that $\text{var}_2(\phi_m)\le \text{var}_2(\phi)$. Moreover, the potential $\phi_m$ is locally constant (in particular, weakly H\"older) and $\limsup_{n\to\infty}\int \phi_m d\mu_n<\infty$.  Let $V=\sum_{k\in\N}\frac{1}{k}1_{[k]}$. Note that $\phi_m+tV$ satisfies the hypothesis in Proposition \ref{prop:5}, and therefore 
$$\limsup_{n\to\infty}\big(h_{\mu_n}(\sigma)-\int (\phi_m+tV) d\mu_n\big)\le \lambda\big(h_\mu(\sigma)-\int (\phi_m+tV) d\mu\big)+(1-\lambda)P(-\phi_m-tV).$$
Since $\lim_{n\to\infty}\int Vd\mu_n=\lambda \int Vd\mu$, we obtain that
$$\limsup_{n\to\infty}\big(h_{\mu_n}(\sigma)-\int \phi_m d\mu_n\big)\le \lambda\big(h_\mu(\sigma)-\int \phi_m d\mu\big)+(1-\lambda)P(-\phi_m-tV).$$
It follows from Proposition \ref{inf} that 
$$\limsup_{n\to\infty}\big(h_{\mu_n}(\sigma)-\int \phi_m d\mu_n\big)\le \lambda\big(h_\mu(\sigma)-\int \phi_m d\mu\big)+(1-\lambda)P_\infty(-\phi_m).$$
Finally, use that $\|\phi-\phi_m\|_0\to0$ to conclude the desired inequality.
\end{proof}

\begin{remark}\label{rem:valeric} Let $(\Sigma,\sigma)$ be a transitive CMS. In Remark \ref{mixtotra} we described the spectral decomposition of a transitive CMS; in this paragraph we will follow the notation introduced in Remark \ref{mixtotra}.  Recall that $\Sigma=\bigcup_{i=1}^p \Sigma_i$, where $p$ is the period of $\Sigma$, $\sigma(\Sigma_i)=\Sigma_{i+1},$ $\Sigma_{p+1}=\Sigma_1$, and $(\Sigma_i,\sigma^p)$ is mixing. We induce on $\Sigma_1\subseteq \Sigma$, where the first return time map is constant with values equal to $p$. It worth noticing that $(\Sigma,\sigma)$ can be identified with the CMS constructed from the pair $(\Sigma_1,\sigma^p)$ and $\tau:\Sigma_1\to\N$, where $\tau$  is constant equal to $p$. Note that $h_{\bar{\mu}}(\sigma^p)+\int \overline{\phi}d\bar{\mu}=p(h_\mu(\sigma)+\int \phi d\mu)$, where $\bar{\mu}$ is the measure in $\M(\Sigma_1)$ corresponding to $\mu\in \M(\sigma)$, and that a sequence $(\mu_n)_n$ in $\M(\sigma)$ converges on cylinders to $\lambda\mu\in\M_{\le 1}(\sigma)$ if and only if $(\bar{\mu}_n)_n$ converges on cylinders to $\lambda\bar{\mu}\in\M_{\le1}(\sigma^p)$. It follows that  $P_\infty(\overline{\phi})=pP_\infty(\phi)$. Since $(\Sigma_1,\sigma^p)$ is mixing, we can apply Proposition \ref{theo:valecul} to the sequence $(\bar{\mu}_n)_n$ and deduce the desired conclusion for the sequence $(\mu_n)_n$.
\end{remark} 

\begin{lemma}\label{lem:valeo}  Let $(\Sigma,\sigma)$ be a transitive CMS and $\phi\in C_{uc}(\Sigma)$ a potential such that $P(-\phi)$ and $\emph{var}_2(\phi)$ are finite.  Let $(\mu_n)_n$ be a sequence  in $\M_{-\phi}(\sigma)$ which converges on cylinders to $\lambda\mu$, where $\lambda\in [0,1]$ and $\mu\in\M_{-\phi}(\sigma)$. Suppose that $s_\infty(-\phi)<1$. Then,
$$\limsup_{n\to\infty}\big(h_{\mu_n}(\sigma)-\int \phi d\mu_n\big)\le \lambda\big(h_\mu(\sigma)-\int \phi d\mu\big)+(1-\lambda)P_\infty(-\phi).$$
\end{lemma}
\begin{proof} Maybe after passing to a subsequence that maximizes the LHS, we can assume that $(\int \phi d\mu_n)_n$ converges to a real number, or to $\infty$. If $\lim_{n\to\infty}\int \phi d\mu_n<\infty$, the result follows from Proposition \ref{theo:valecul}. If $\lim_{n\to\infty}\int \phi d\mu_n=\infty$, then $\lim_{n\to\infty}h_{\mu_n}(\sigma)-\int \phi d\mu_n=-\infty$ (see Lemma \ref{lem:-inf}), in which case the inequality trivially holds.
\end{proof}

\begin{proof}[Proof of Theorem \ref{thm:1}] The proof of the inequality follows by  Proposition \ref{theo:valecul} and Lemma \ref{lem:valeo} for the potential $-\phi$. To see that the inequality is sharp, consider a sequence $(\nu_n)_n$ in $\M_\phi(\sigma)$ which converges to the zero measure and such that $\lim_{n\to\infty}\big(h_{\mu_n}(\sigma)+\int \phi d\mu_n\big)=P_\infty(\phi)$. Set $\eta_n=\lambda \mu+(1-\lambda)\nu_n$. Note that $(\eta_n)_n$ converges to $\lambda\mu$ and that it realizes the equality.
\end{proof}

\begin{corollary}\label{cor:finent} Let $(\Sigma,\sigma)$ be a transitive CMS with finite entropy and $\phi\in C_{b,uc}(\Sigma)$.  Let $(\mu_n)_n$ be a sequence  in $\M(\sigma)$ which converges on cylinders to $\lambda\mu$, where $\lambda\in [0,1]$ and $\mu\in\M(\sigma)$. Then, $\limsup_{n\to\infty}\big(h_{\mu_n}(\sigma)+\int \phi d\mu_n\big)\le \lambda\big(h_\mu(\sigma)+\int \phi d\mu\big)+(1-\lambda)P_\infty(\phi).$
\end{corollary}

\section{Existence of equilibrium states and maximizing measures}\label{maxme}
In this section we derive applications of Theorem \ref{thm:bsctm} and Theorem \ref{thm:1} to the problem of existence and non-existence of equilibrium states and maximizing measures. Recall that we have defined $\H=\{\phi\in C_{uc}(\Sigma): \sup(\phi), \text{var}_2(\phi),P(\phi)<\infty\}$.

\subsection{Existence and non-existence of equilibrium states} We will now prove Theorem \ref{thm:2}, which provides a criterion for the existence of equilibrium states and describes what occurs when no equilibrium state exists. 

\begin{proof}[Proof of Theorem \ref{thm:2}] 
\noindent\\\\
(\ref{122}): Observe that if $s_\infty(\phi)<1,$ and $\liminf_{n\to\infty}\int \phi d\mu_n=-\infty$, then $\liminf_{n\to\infty} \big(h_{\mu_n}(\sigma)+\int \phi d\mu_n\big)=-\infty$ (see Lemma \ref{lem:-inf}), which is not possible. We can therefore assume that $\liminf_{n\to\infty}\int \phi  d\mu_n>-\infty$. It follows by Theorem \ref{thm:bsctm} that $(\mu_n)_n$ has a subsequence that converges on cylinders to $\lambda\mu\in\M_{\le 1}(\sigma)$, where $\lambda\in [0,1]$ and $\mu\in \M_{\phi}(\sigma)$. By Theorem \ref{thm:1},
$$\limsup_{n\to\infty}\big(h_{\mu_n}(\sigma)+\int \phi d\mu_n\big)\le \lambda\big(h_\mu(\sigma)+\int \phi d\mu\big)+(1-\lambda)P_\infty(\phi).$$
Therefore, $P(\phi)= \lambda\big(h_\mu(\sigma)+\int \phi d\mu\big)+(1-\lambda)P_\infty(\phi)$. Finally, note that $P_\infty(\phi)<P(\phi)$ implies that $\lambda=1$, and therefore $\mu$ is an equilibrium state.\\\\
(\ref{222}): If there is no subsequence for which the integral goes to $-\infty$, then we have that  $\liminf_{n \to \infty} \int \phi \, d\mu_n > -\infty$.  In this case, Theorem \ref{thm:bsctm} implies that $(\mu_n)_n$ has a subsequence that converges on cylinders to a measure $\lambda\mu\in\M_{\le 1}(\sigma)$, where $\lambda\in [0,1]$ and $\mu\in \M_{\phi}(\sigma)$. By Theorem \ref{thm:1}, we have that 
$$\limsup_{n\to\infty}\big(h_{\mu_n}(\sigma)+\int \phi d\mu_n\big)\le \lambda\big(h_\mu(\sigma)+\int \phi d\mu\big)+(1-\lambda)P_\infty(\phi).$$
Therefore, $P(\phi)=\lambda\big(h_\mu(\sigma)+\int \phi d\mu\big)+(1-\lambda)P_\infty(\phi)$. Since $h_\mu(\sigma)+\int \phi d\mu<P(\phi)$, we conclude that $\lambda=0$
\end{proof}

\begin{remark} Let $\phi\in \H$ and define $f(t)=P(t\phi)$. If $t>s_\infty(\phi)$ and $t\phi$ is SPR, then $t\phi$  has an equilibrium state (note that  $s_\infty(t\phi)<1$). If $\phi$ is SPR and $s_\infty(\phi)=1$, then it admits an equilibrium state if and only if the slopes of the supporting lines for the graph of $f$ at $(t,f(t))$ remain bounded as $t\to 1^+$.
\end{remark}


\begin{remark} 
Let $\phi\in \H$ be a potential which does not have an equilibrium state. Let $(\mu_n)_n$ be a sequence in $\M_{\phi}(\sigma)$ such that $P(\phi)=\lim_{n\to\infty} \big(h_{\mu_n}(\sigma)+\int \phi d\mu_n\big).$ Then:
\begin{enumerate}
\item If $(\Sigma,\sigma)$ has finite topological entropy, then $(\mu_n)_n$ converges to the zero measure. 
\item If $(\Sigma,\sigma)$ is the full shift on a countable alphabet and $\text{var}_1(\phi)<\infty$, then \\ $\lim_{n\to\infty}\int \phi d\mu_n=-\infty$ (see Lemma \ref{lem:vale1}).
\item If $P_\infty(\phi)=-\infty$, then $(\mu_n)_n$ does not have any subsequence converging to the zero measure (otherwise, $P_\infty(\phi)=P(\phi)$). Therefore, 
$\lim_{n\to\infty}\int \phi d\mu_n=-\infty$.
\end{enumerate}
\end{remark}

\begin{example}\label{addex} Let $(\Sigma,\sigma)$ be a null recurrent finite entropy CMS. As described in subsection  \ref{sec:spr}, the induced transformation over a $1$-cylinder $[a]$ is conjugate to the full shift $(\N^\N,\sigma_f)$. Let $\tau:\N^\N\to\N$ be the first return time map and set $f(t)=P(-t\tau)$. Note that $\tau$ depends only on the first coordinate, that is,  $\text{var}_1(\tau)=0$.  Since $(\Sigma,\sigma)$ null recurrent, we have that $f(t)=\infty$ if $t<h_{top}(\sigma)$ and $f(h_{top}(\sigma))=0$. The potential $\phi:=-h_{top}(\sigma)\tau$ is positive recurrent (see \cite[Corollary 2]{sap}; see also Proposition \ref{lem:finalfull}). Let $\mu_\phi$ be the RPF measure of $\phi$. Since $(\Sigma,\sigma)$ has no measure of maximal entropy, we have that $\int \tau d\mu_\phi=\infty$, or $\int \phi d\mu_\phi=-\infty$; equivalently, $\phi$ does not have an equilibrium state in the classical sense. For $t>1,$ the potential $t\phi$ has a unique equilibrium state, which we denote by $\mu_t$. Then,   $$0=P(\phi)=\lim_{t\to1^+}\big( h_{\mu_t}(\sigma)+\int \phi d\mu_t\big),\text{ and }\lim_{t\to 1^+}\int \phi d\mu_t=-\infty.$$
In this example, we observe that the finiteness  condition on the integral of the sequence of measures cannot be removed from the assumptions in Theorem \ref{thm:2}(\ref{122}) in order to guarantee the existence of an equilibrium state.
\end{example}


\subsection{Ergodic optimization} For $\phi \in C(\Sigma)$ we define $$\beta(\phi )=\sup_{\mu\in \M(\sigma)}\int \phi d\mu.$$ We say that $\mu\in \M(\sigma)$ is a maximizing measure of $\phi $ if $\beta(\phi )=\int \phi d\mu$. As mentioned in the introduction, we consider a quantity analogous to the pressure at infinity. Define
$$\beta_\infty(\phi)=\sup_{(\mu_n)_n\to 0}\limsup_{n\to\infty} \int \phi d\mu_n,$$
where the supremum runs over sequences in $\M(\sigma)$ that converge on cylinders to the zero measure. If there is no such sequence we define $\beta_\infty(\phi)=-\infty$.  Note that  $P_\infty(\phi)\ge \beta_\infty(\phi)$.

Continuous potentials on subshifts of finite type admit maximizing measures. This is not generally the case for countable Markov shifts, and we are interested in establishing a criterion for the existence of a maximizing measure in this setting. It is fairly easy to construct potentials without maximizing measures, as we can see in the example below.

\begin{example}\label{ex:nomea} Let $(\Sigma,\sigma)$ be a CMS for which there exists a sequence $(\mu_n)_n$ in $\M(\sigma)$ which converges on cylinders to the zero measure (for instance,  if $h_{top}(\sigma)$ is finite; see also Theorem \ref{cry}). Let $\phi=1-\sum_{k=1}^\infty \frac{1}{k}1_{[k]}$. Note that $\lim_{n\to\infty}\int \phi d\mu_n=1$, and therefore $\beta(\phi)\ge 1.$ Since $\phi<1$, we conclude that $\beta(\phi)=1$, and that $\phi$ does not have any maximizing measure. Furthermore, note that $\beta_\infty(\phi)=1$. 
\end{example}

We will prove a result that describes the limiting behavior of the integral of a potential when escape of mass is allows, which is analogous to Theorem \ref{thm:1} and improves Lemma \ref{tired}. For this, we will use the fact that  $\beta_\infty$ is the asymptotic slope of the map $t\mapsto P_\infty(t\phi)$.

\begin{lemma}\label{prop:betaas}
Let $\phi\in \H$. Then $$\beta_\infty(\phi)=\lim_{t\to\infty}\frac{1}{t}P_\infty(t\phi).$$
\end{lemma}
\begin{proof} First, note that $P_\infty(t\phi)\ge \beta_\infty(t\phi)=t\beta_\infty(\phi)$, and therefore $\liminf_{t\to\infty}\frac{1}{t}P_\infty(t\phi)\ge \beta_\infty(\phi)$. Set $A=\limsup_{t\to\infty}\frac{1}{t}P_\infty(t\phi)$. Observe that $P_\infty(t\phi)\le tP_\infty(\phi)$, for all $t\ge 1$, and therefore $A\le P_\infty(\phi)<\infty$. We claim that $A\le \beta_\infty(\phi)$. We will assume that $P_\infty(t\phi)>-\infty$, for all $t>0$, otherwise $A=\beta_\infty(\phi)=-\infty$. 

Let $(t_n)_n$ be a sequence going to infinity such that $\lim_{n\to\infty}\frac{1}{t_n}P_\infty(t_n\phi)=A$. Denote by $\nu_0$ the zero measure. Let $(\nu_n)_n$ be a sequence in $\M(\sigma)$ such that $\rho(\nu_n,\nu_0)\le \frac{1}{n}$ (see (\ref{metricc})), and $\big(h_{\nu_n}(\sigma)+t_n\int \phi d\nu_n\big)\in [P_\infty(t_n\phi)-\frac{1}{n},P_\infty(t_n\phi)+\frac{1}{n}]$. We conclude that 
$$\lim_{n\to\infty}\bigg(\frac{1}{t_n}h_{\nu_n}(\sigma)+\int \phi d\nu_n\bigg)=A,$$
and that $(\nu_n)_n$ converges to the zero measure. Note that 
$$\frac{t_n-1}{t_n}h_{\nu_n}(\sigma)+\bigg(\frac{1}{t_n}h_{\nu_n}(\sigma)+\int \phi d\nu_n\bigg)\le P(\phi),$$
and therefore $\limsup_{n\to\infty}h_{\nu_n}(\sigma)$ is finite. We obtain that $\lim_{n\to\infty}\frac{1}{t_n}h_{\nu_n}(\sigma)=0$. Finally, note that $A=\lim_{n\to\infty}\int \phi d\nu_n\le \beta_\infty(\phi)$, which concludes the proof of the lemma. \end{proof}

\begin{remark}\label{what} It is known that if $\phi\in \H$, then $\beta(\phi)=\lim_{t\to\infty}\frac{1}{t}P(t\phi)$ (the proof is analogous to the proof of Lemma \ref{prop:betaas}). 
\end{remark}

\begin{theorem}\label{thm:maxm}  Let $(\Sigma,\sigma)$ be a transitive CMS and $\phi\in\H$.  Let $(\mu_n)_n$ be a sequence  in $\M(\sigma)$ which converges on cylinders to $\lambda\mu$, where $\lambda\in [0,1]$ and $\mu\in\M(\sigma)$. Then,
$$\limsup_{n\to\infty} \int \phi d\mu_n \le \lambda\int \phi d\mu+(1-\lambda)\beta_\infty(\phi).$$
Moreover, the inequality is sharp.
\end{theorem}
\begin{proof} Maybe after passing to a subsequence we can assume that $\lim_{n\to\infty}\int \phi d \mu_n$ is well defined. By Lemma \ref{tired}, we have that $\int \phi d\mu>-\infty$. By Theorem \ref{thm:1}, applied to the potential $t\phi$, for $t\ge 1$, and obtain that
$$\limsup_{n\to\infty}\bigg(\frac{1}{t}h_{\mu_n}(\sigma)+\int \phi d\mu_n\bigg)\le \lambda\bigg(\frac{1}{t}h_\mu(\sigma)+\int \phi d\mu\bigg)+(1-\lambda)\frac{1}{t}P_\infty(t\phi).$$
Sending $t$ to infinity we obtain the desired inequality (see Lemma \ref{prop:betaas}). To see that the inequality is sharp, consider a sequence $(\nu_n)_n$ converging to the zero measure and such that $\lim_{n\to\infty}\int \phi d\nu_n=\beta(\phi)$ and define $\eta_n=\lambda\mu+(1-\lambda)\nu_n$. The sequence $(\eta_n)_n$ converges to $\lambda\mu$ and achieves equality in the inequality above.
\end{proof}

The next result is analogous to Theorem \ref{thm:2} and the proof is essentially the same, considering Theorem \ref{thm:maxm} instead of Theorem \ref{thm:1}. We leave the details to the reader.

\begin{theorem}\label{maxva} Let $(\Sigma,\sigma)$ be a transitive CMS and $\phi\in \H$.  Let $(\mu_n)_n$ be a sequence in $\M(\sigma)$ such that $\lim_{n\to\infty}\int \phi d\mu_n=\beta(\phi)$. Then:
\begin{enumerate}
\item If $\beta_\infty(\phi)<\beta(\phi)$, then $(\mu_n)_n$ has a subsequence which converges weak$^*$ to a maximizing measure of $\phi$. In particular, $\phi$ admits a maximizing measure.  
\item If $\phi$ does not have a maximizing measure, then $(\mu_n)_n$ converges on cylinders to the zero measure.
\end{enumerate}
\end{theorem}

Note that Theorem \ref{thm:mm} is a direct consequence of Theorem \ref{maxva}.  We finish this subsection with a result about zero temperature limits for a suitable class of potentials.

\begin{proposition}\label{fv1}  Let $(\Sigma,\sigma)$ be a transitive CMS and $\phi\in \H$ a potential with summable variations such that $\beta_\infty(\phi)<\beta(\phi)$. Then, there exists $t_0\in [0,\infty)$ such that $t\phi$ has a unique equilibrium state $\mu_t$, for every $t\in (t_0,\infty)$. Moreover, if $(t_n)_n$ goes to $\infty$, then $(\mu_{t_n})_n$ has a subsequence which converges weak$^*$ to a maximizing measure of $\phi$. 
\end{proposition}

\begin{proof} Since $\beta_\infty(\phi)=\lim_{t\to\infty}\frac{1}{t}P_\infty(t\phi),$ and $\beta_\infty(\phi)<\beta(\phi)$, it follows that $P_\infty(t\phi)< t\beta(\phi)\le P(t\phi)$ for large enough $t$. Therefore, there exists $t_0\ge 1$ such that if $t>t_0$, then $t\phi$ is SPR. Since $s_\infty(t\phi)=\frac{1}{t}s_\infty(\phi)<1$, for all $t>t_0$, we conclude that $t\phi$ has an equilibrium state (see Theorem \ref{thm:2}). Since $t\phi$ has summable variations, the equilibrium state is unique.   Note that $P'(t\phi)=\int \phi d\mu_t$, for every $t>t_0$ (see \cite{sa2}). By convexity of the map $t\mapsto P(t\phi)$, $\lim_{t\to\infty}P'(t\phi)=\lim_{t\to\infty}\frac{1}{t}P(t\phi)=\beta(\phi)$ (see Remark \ref{what}). Then, $\lim_{t\to\infty}\int \phi d\mu_t=\beta(\phi)$, and the result follows from Theorem \ref{maxva}.
\end{proof}

\section{Pressure at infinity of suspension flows}\label{sec:flow}
 Let $(\Sigma, \sigma)$ be a transitive CMS and $\tau$ a continuous potential  bounded away from zero, that is, there exists $c=c(\tau)>0$ such that $\tau(x)\ge c$, for every $x\in \Sigma$.  Define
\begin{equation*} 
Y= \left\{ (x,t)\in \Sigma  \times \R \colon 0 \le t \le\tau(x) \right\}/\sim,
\end{equation*}
where we identify the points $(x,\tau(x))$ and $(\sigma(x),0)$, for all $x\in \Sigma $. The suspension flow over $\sigma$
with roof function $\tau$ is the semi-flow $\Theta = (
\theta_t)_{t \ge 0}$ on $Y$ defined by
\begin{equation*}
 \theta_t(x,s)= (x,s+t), \ \text{whenever $s+t\in[0,\tau(x)]$.}
\end{equation*} 

Denote by $\M(\Theta)$  the space of $\Theta$-invariant probability measures on $Y$ and $\M_{\le1}(\Theta)$  the space $\Theta$-invariant sub-probability measures on $Y$. Let $\M(\sigma;\tau)=\{\mu\in\M(\sigma): \int\tau d\mu<\infty\}$. It follows from a classical result of Ambrose and Kakutani \cite{ak} that the map $\varphi: \M(\sigma;\tau) \to \M(\Theta),$ defined by $\varphi(\mu)= \frac{\mu \times Leb}{\int \tau d\mu},$ 
where $Leb$ is the one dimensional Lebesgue measure, is a bijection. In particular, a measure $\nu\in \M_{\le 1}(\Theta)$ different from the zero measure can be uniquely  written as $\nu=\lambda  \frac{\mu \times Leb}{\int \tau d\mu}$, for some $\mu\in \M(\sigma;\tau)$ and $\lambda\in (0,1]$. 

Recall that $\Psi=\{\psi\in C_{uc}(\Sigma): \inf \psi>0, \text{var}_2(\psi)<\infty, P(-\psi)<\infty\}.$ For our applications we will be mainly interested in the case where $\tau\in \Psi$. 

\subsection{Topology of convergence on cylinders} In \cite[Section 6]{iv}, the authors introduced a topology on the space $\M_{\le 1}(\Theta)$ which is analogous to the cylinder topology on $\M_{\le1}(\sigma)$. 

\begin{definition} \label{def:top1} Let $(\nu_n)_n$ and $\nu$ be measures in $\M_{\le 1}(\Theta)$.  We say that $(\nu_n)_n$ converges on cylinders to $\nu$ if $$\lim_{n\to\infty}\nu_n(C\times [0,c])=\nu(C\times[0,c]),$$
for every cylinder $C\subseteq \Sigma$, and $c=c(\tau)$. 
\end{definition}

This notion of convergence induces the cylinder topology on $\M_{\le 1}(\Theta)$. Let $(C_i)_{i\in\N}$ be an enumeration of the cylinders of $\Sigma$. The function
$\rho_\tau:\M_{\le1}(\Theta)\times \M_{\le1}(\Theta)\to\R_{\ge 0},$ given by 
$$ \rho_\tau(\nu_1,\nu_2)=\sum_{i\in \N}\frac{1}{2^i}|\nu_1(C_i\times [0,c])-\nu_2(C_i\times [0,c])|,$$
is a metric on $\M_{\le1}(\Theta)$, compatible with the cylinder topology (see \cite[Lemma 6.6]{iv}). A sequence $(\nu_n)_n$ in $\M(\Theta)$ converges on cylinders to $\nu\in \M(\Theta)$ if and only if $(\nu_n)_n$ converges  in the weak$^*$ topology to $\nu$ (see \cite[Lemma 6.7]{iv}). In other words, the cylinder topology on $\M_{\le1}(\Theta)$ induces the weak$^*$ topology on  $\M(\Theta)$.

 By Kac's formula we have that $\nu(C\times [0,c])=\frac{c\mu(C)}{\int \tau d\mu},$ whenever $\nu\in \M(\Theta)$ and $\nu=\varphi(\mu)$. In particular, a sequence $(\nu_n)_n$ in $\M(\Theta)$ converges on cylinders to $\lambda \nu$, where $\lambda\in [0,1]$ and $\nu\in \M(\Theta)$, if and only if
 \begin{align}\label{valeric} \lim_{n\to\infty}\frac{\mu_n(C)}{\int \tau d\mu_n}=\lambda\frac{\mu(C)}{\int \tau d\mu},\end{align}
for every cylinder $C\subseteq \Sigma$, where $\varphi(\mu_n)=\nu_n$ and $\varphi(\mu)=\nu$ (see \cite[Remark 6.5]{iv}). A more detailed analysis on the relation given by (\ref{valeric}) give us the following useful fact.

\begin{lemma}\label{topsv} Let $(\mu_n)_n, \mu$ be measures in $\M(\sigma;\tau)$ and set $\nu_n=\varphi^{-1}(\mu_n)$, $\nu=\varphi^{-1}(\mu)$. Then the following staments are equivalent:
\begin{enumerate}
\item  $(\nu_n)_n$ converges on cylinders to the zero measure. 
\item Every subsequence of $(\mu_n)_n$ has a subsubsequence $(\mu_{n_k})_k$ which converges on cylinders to the zero measure or 
such that $\lim_{k\to\infty}\int \tau d\mu_{n_k}=\infty$.
\end{enumerate}
Similarly, the following statements are equivalent:
\begin{enumerate}
\item  $(\nu_n)_n$ converges on cylinders to $\lambda\nu$, where $\lambda\in (0,1]$. 
\item Every subsequence of $(\mu_n)_n$ has a subsubsequence $(\mu_{n_k})_k$ which converges on cylinders to $\lambda_1\mu$ and 
 $\lim_{k\to\infty}\int \tau d\mu_{n_k}=\lambda_2\int \tau d\mu$, for some $\lambda_1\in (0,1]$ and $\lambda_2\in [1,\infty)$ satisfying   $\lambda=\lambda_1/\lambda_2$.
\end{enumerate}
\end{lemma}

\begin{remark} It follows by Lemma \ref{topsv} that $(\nu_n)_n$ converges weak$^*$ to $\nu$ if and only if $(\mu_n)_n$ converges weak$^*$ to $\mu$ and $\lim_{n\to\infty} \int\tau d\mu_n=\int \tau d\mu$, where $\nu_n=\varphi^{-1}(\mu_n)$, $\nu=\varphi^{-1}(\mu)$.
\end{remark}

We now provide a proof of Theorem \ref{compactv}, which is analogous to Theorem \ref{compact} in the finite entropy case.

\begin{proof}[Proof of Theorem \ref{compactv}] Let $(\nu_n)_n$ be a sequence in $\M(\Theta)$ and set $\mu_n=\varphi^{-1}(\nu_n)$. We will prove that $(\nu_n)_n$ has a convergent subsequence. Maybe after passing to a subsequence we can assume that $(\int \tau d\mu_n)_n$ converges to $\infty$, or a real number. 
Note that if $\lim_{n\to\infty}\int \tau d\mu_n=\infty$, then by Lemma \ref{topsv}, $(\nu_{n})_n$ converges on cylinders to the zero measure. Let us now suppose that $\lim_{n\to\infty}\int\tau d\mu_n<\infty$. It follows by Theorem \ref{thm:bsctm} (applied to $-\tau$) that there exists a subsequence $(\mu_{n_k})_k$ which converges on cylinders to $\lambda \mu$, where $\lambda\in [0,1]$ and  $\mu\in \M(\sigma;\tau)$. By Lemma \ref{topsv}, the sequence $(\nu_{n_k})_k$ converges on cylinders to a measure in $\M_{\le 1}(\Theta)$. In general, if $(\nu_n)_n$ is any sequence in $\M_{\le1}(\Theta)$, we consider a subsequence where the mass converges and apply the discussion above.

To prove the density of $\M(\Theta)$ in $\M_{\le 1}(\Theta),$ it is enough to prove that there exists a sequence $(\nu_n)_n$ in $\M(\Theta)$ that converges to the zero measure. Indeed, if $\lambda\mu\in \M_{\le 1}(\Theta)$, then we can consider the sequence $(\lambda\nu+(1-\lambda)\nu_n)_n$ which converges to $\lambda\nu$. If $(\Sigma,\sigma)$ has finite entropy, then there exists a sequence $(\mu_n)_n$ which converges to the zero measure (see \cite[Lemma 4.6]{iv}). If $(\Sigma,\sigma)$ has infinite entropy, consider a sequence $(\mu_n)_n$ in $\M(\sigma)$ such that $\lim_{n\to\infty}h_{\mu_n}(\sigma)=\infty$; since $h_{\mu_n}(\sigma)-\int \tau d\mu_n\le P(-\tau)$, we conclude that $\lim_{n\to\infty}\int \tau d\mu_n=\infty$. In both cases it follows by Lemma \ref{topsv} that $(\varphi(\mu_n))_n$ converges to the zero measure. 
\end{proof}

\subsection{Entropy and pressure at infinity for suspension flows}

The entropy of a measure $\nu\in \M(\Theta)$ is denoted  by $h_\nu(\Theta)$ (this is defined as the entropy of $\nu$ with respect to the time one map $\theta_1$).  By Abramov's formula, we have that $h_\nu(\Theta)=\frac{h_\mu(\sigma)}{\int \tau d\mu},$ where $\varphi(\mu)=\nu$.  The entropy of  the suspension flow $(Y,\Theta)$ is defined by $h_{top}(\Theta)=\sup_{\nu\in\M(\Theta)} h_\nu(\Theta).$
Similarly, for a potential $\phi:Y\to\R$ we define its pressure by the formula 
$$P^\Theta(\phi)=\sup_{\nu\in\M(\Theta)} \bigg(h_\nu(\Theta)+\int \phi d\nu\bigg).$$
We added the superscript $\Theta$ to emphasize that the base dynamical system is the suspension flow $(Y,\Theta)$. We define the pressure at infinity of $\phi$ by the formula
$$P^\Theta_\infty(\phi)=\sup_{(\nu_n)_n\to 0} \limsup_{n\to\infty}\bigg( h_{\nu_n}(\Theta)+\int \phi d\nu_n\bigg),$$
where the supremum runs over all sequences $(\nu_n)_n$ in $\M(\Theta)$ which converge on cylinders to the zero measure. If  there is no such sequence, we set $P^\Theta_\infty(\phi)=-\infty$.  The entropy at infinity of $(Y,\Theta)$, which we denote by $h_{\infty}(\Theta)$, is defined as the pressure at infinity of the zero potential. 

For a potential $\phi:Y\to\R$, we define $\Delta_\phi:\Sigma\to\R$ by the formula $\Delta_\phi(x)=\int_0^{\tau(x)}\phi(\theta_t x)dt$. It follows by Kac's formula that  $\int \phi d\nu=\frac{\int \Delta_\phi d\mu}{\int \tau d\mu}$, where $\mu=\varphi^{-1}(\nu)$. Therefore, \begin{align}\label{presusp} P^\Theta(\phi)=\inf\{t: P(\Delta_\phi- t\tau)\le 0\}.\end{align}
We will prove that a similar formula holds for the pressure at infinity.

\begin{lemma}\label{lem:formulapara} $P_\infty^\Theta(\phi)=\inf\{t: P^{top}_\infty(\Delta_\phi- t\tau)\le 0\}.$
\end{lemma}
\begin{proof} Since $\tau$ is bounded away from zero, the function $f(t)=P^{top}_\infty(\Delta_\phi- t\tau)$ is decreasing. Define $s=\inf\{t: P^{top}_\infty(\Delta_\phi- t\tau)<\infty\}$ and set $a=\inf\{t: P^{top}_\infty(\Delta_\phi- t\tau)\le 0\}$. Note that $f(t)=\infty$ if $t<s$ and $f(t)<\infty$ if $t>s$. 

If there is no sequence in $\M(\Theta)$ that converges to the zero measure, then there is no sequence in $\M(\sigma)$ that converges to the zero measure (see Lemma \ref{topsv}). In this case, we have that $P_\infty^\Theta(\phi) = -\infty,$ and that $P^{top}_\infty(\Delta_\phi - t\tau) = -\infty$, for all $t$ (see Remark \ref{rem:inf}), and thus the formula holds. Let us now assume that there exists a sequence in $\M(\Theta)$ which converges to the zero measure. In particular, $P_\infty(\phi)$ can be realized by a sequence of measures converging to the zero measure.

Fix $t>a$. Let $(\nu_n)_n$ be a sequence in $\M(\Theta)$ that converges to the zero measure and such that $P_\infty^\Theta(\phi)=\lim_{n\to\infty}\big(h_{\nu_n}(\Theta)+\int \phi d\nu_n\big)$. Set $\mu_n\in \M(\sigma)$ such that $\varphi(\mu_n)=\nu_n$. Maybe after passing to a subsequence, we can assume that $(\int \tau d\mu_n)_n$ converges to a real number or $\infty$. In the first case, we have that $(\mu_n)_n$ converges to the zero measure (see Lemma \ref{topsv}), and therefore $\limsup_{n\to\infty}\big(h_{\mu_n}(\sigma)+\int (\Delta_\phi-t\tau) d\mu_n\big)< 0$, which implies that $\lim_{n\to\infty}\big(h_{\nu_n}(\Theta)+\int \phi d\nu_n\big)< t$. In the second case, consider $\epsilon>0$ and note that $$h_{\mu_n}(\sigma)+\int (\Delta_\phi-t\tau) d\mu_n\le P(\Delta_\phi- (t+\epsilon)\tau)+\epsilon\int \tau d\mu_n,$$ which implies that $\lim_{n\to\infty}h_{\nu_n}(\Theta)+\int \phi d\nu_n\le t+\epsilon$ (observe that $P(\Delta_\phi- (t+\epsilon)\tau)<\infty$,  because $P^{top}_\infty(\Delta_\phi- (t+\epsilon)\tau)<\infty$, see Remark \ref{rem:inf}(1)). In both cases we conclude that $P_\infty^\Theta(\phi)<t$, for every $t>a$, and therefore $P_\infty^\Theta(\phi)\le a$. 

It remains to prove that $P_\infty^\Theta(\phi)\ge a$. We consider two cases:\\
Case 1 ($s<a$): In this case, we have that $P^{top}_\infty(\Delta_\phi-a\tau)=0$. Therefore, there exists a sequence $(\mu_n)_n$ that converges on cylinders to the zero measure such that $\lim_{n\to\infty}\big(h_{\mu_n}(\sigma)+\int (\Delta_\phi-a\tau)d\mu_n\big)=0$, and thus $\lim_{n\to\infty}\big(h_{\nu_n}(\Theta)+\int \phi d\nu_n\big)=a$, where $\nu_n=\varphi(\mu_n)$. The sequence $(\nu_n)_n$ converges to zero, and therefore $P_\infty^\Theta(\phi)\ge a$.\\
Case 2 ($s=a$): Fix $\epsilon>0$. Since  $P^{top}_\infty(\Delta_\phi-(s-\epsilon)\tau)=\infty$, it follows that $P(\Delta_\phi-(s-\epsilon)\tau)=\infty$. There exists a sequence $(\mu_n)_n$ in $\M(\sigma)$ such that 
$$\lim_{n\to\infty}\big(h_{\mu_n}(\sigma)+\int (\Delta_\phi-(s-\epsilon)\tau)d\mu_n\big)=\infty.$$ Then, $\limsup_{n\to\infty}\big(h_{\nu_n}(\Theta)+\int \phi d\nu_n\big)\ge s-\epsilon$, where $\nu_n=\varphi(\mu_n)$. Note that $$h_{\mu_n}(\sigma)+\int (\Delta_\phi-(s-\epsilon)\tau)d\mu_n\le P(\Delta_\phi-(s+\epsilon)\tau)+2\epsilon \int \tau d\mu_n,$$ and therefore $\lim_{n\to\infty} \int \tau d\mu_n=\infty$. We conclude that $(\nu_n)_n$ converges to zero, and then $P_\infty^\Theta(\phi)\ge s-\epsilon$. Since $\epsilon>0$ was arbitrary, it follows that  $P_\infty^\Theta(\phi)\ge s$.
\end{proof}



\begin{remark}\label{Psisusp} Let $\tau\in \Psi$. Note that $h_{top}(\Theta)<\infty$ (see equation (\ref{presusp})). In particular, if $\phi\in C_b(Y)$, then $P^\Theta(\phi)<\infty$. 
Furthermore, the entropy at infinity of $(Y,\Theta)$ is given by $h_\infty(\Theta)=h(\tau)$ (see Definition \ref{def:123} and Lemma \ref{lem:formulapara}).
\end{remark}

We now provide a proof of Theorem \ref{thm:susp}, which is similar to Theorem \ref{thm:1} and analogous to Corollary \ref{cor:finent}, where we considered a finite entropy CMS and a bounded potential.

\begin{proof}[Proof of Theorem \ref{thm:susp}] Maybe after passing to a subsequence we can assume that $(h_{\nu_n}(\Theta)+\int \phi d\nu_n)_n$ is convergent and converges to the limsup. Set $\mu_n=\varphi^{-1}(\nu_n)$ and $\mu=\varphi^{-1}(\nu)$. Assume that $\lambda\in (0,1]$, otherwise the result follows by definition of the pressure at infinity. By Lemma \ref{topsv}, maybe after passing to a subsequence, we can assume that $(\mu_{n})_n$  converges on cylinders to $\lambda_1\mu$, and $\lim_{k\to\infty}\int \tau d\mu_{n}=\lambda_2\int \tau d\mu$, where $\lambda_1\in (0,1]$, $\lambda_2\in [1,\infty)$ and $\lambda=\lambda_1/\lambda_2$. 

By Lemma \ref{lem:formulapara}, we have that if $t>P_\infty^\Theta(\phi)$, then $P^{top}_\infty(\Delta_\phi-t\tau)<0$. Note that $\Delta_\phi-t\tau$ is uniformly continuous with finite second variation and finite pressure (see Lemma \ref{lem:infpre}). In particular, we can apply Theorem \ref{thm:1} and conclude that 
 $$\limsup_{n\to\infty}\big(h_{\mu_{n}}(\sigma)-\int (\Delta_\phi-t\tau)  d\mu_{n}\big)\le \lambda_1\big(h_{\mu}(\sigma)- \int (\Delta_\phi-t\tau) d\mu\big),$$
and therefore, 
\begin{align*}
\limsup_{n\to\infty} (h_{\nu_n}(\Theta)-\int (\phi-t) d\nu_n)&=\limsup_{n\to\infty}\bigg(\frac{h_{\mu_{n}}(\sigma)-\int (\Delta_\phi-t\tau)  d\mu_{n}}{\int \tau d\mu_{n}}\bigg)\\
&\le \frac{\lambda_1}{\lambda_2}\bigg(\frac{h_{\mu}(\sigma)- \int (\Delta_\phi-t\tau) d\mu}{\int \tau d\mu}\bigg)\\
&=\lambda(h_\nu(\Theta)- \int (\phi-t)d\nu),
\end{align*}
which is equivalent to 
$$\limsup_{n\to\infty} \bigg(h_{\nu_n}(\Theta)+\int \phi d\nu_n\bigg)\le \lambda (h_{\nu}(\Theta)+\int\phi d\nu)+(1-\lambda)t.$$
Since $t>P_\infty^\Theta(\phi)$ was arbitrary we obtain the desired inequality. 

We now prove that the inequality is sharp. Let $(\eta_n)_n$ be a sequence in $\M(\Theta)$ that converges to the zero measure and such that $\lim_{n\to\infty} \big(h_{\eta_n}(\Theta)+\int \phi d\eta_n\big)=P_\infty^\Theta(\phi)$. The existence of such sequence follows by Theorem \ref{compactv}. Set $\nu_n=\lambda\nu+(1-\lambda)\eta_n.$ Note that $(\nu_n)_n$ converges to $\lambda\nu$ and it achives equality in the pressure inequality. 
\end{proof}

Analogous to the case of CMS, we say that $\phi:Y\to\R$ is SPR if  $P^\Theta(\phi)>P_\infty^\Theta(\phi)$.  The next result follows from combining  Theorem \ref{compactv} and Theorem \ref{thm:susp}. The proof is analogous to the proof of Theorem \ref{thm:2} and we leave the details to the reader. 

\begin{theorem}\label{fine:sus} Let $\tau\in \Psi$ and $\phi\in \H_Y$. Let $(\nu_n)_n$ be a sequence in $\M(\Theta)$ such that  $P^\Theta(\phi)=\lim_{n\to\infty} \big( h_{\nu_n}(\Theta)+\int \phi d\nu_n\big)$. Then:
\begin{enumerate}
\item If $\phi$ is SPR, then $(\nu_n)_n$ has a subsequence which converges in the weak$^*$ topology to an equilibrium state of $\phi$. In particular, $\phi$ admits equilibrium states.
\item If $\phi$ does not have an equilibrium state, then $(\nu_n)_n$ converges to the zero measures. In this case, $P_\infty^\Theta(\phi)=P^\Theta(\phi)$.
\end{enumerate}

\end{theorem}

\appendix

\section{Tightness in the space of invariant probability measures}

Let $(\Sigma,\sigma)$ be a transitive countable Markov shift. Let $(\mu_n)_n$ be a sequence in $\M(\sigma)$. The condition
\begin{align}\label{pivale}\lim_{k\to\infty}\limsup_{n\to\infty}\mu_n\big(\bigcup_{s\ge k}[s]\big)=0,\end{align}
says that $(\mu_n)_n$ does not lose mass in the cylinder topology. In Lemma \ref{lem:nicev}, we prove that this condition is equivalent to saying that  $(\mu_n)_n$ is tight. Since $\Sigma$ may not be locally compact, it is convenient to replace tightness with something more suitable to our setup. 

\begin{lemma}\label{lem:nicev} Let $(\mu_n)_n$ be a sequence in $\M(\sigma)$ such that (\ref{pivale}) holds. Then, there exists a subsequence of $(\mu_n)_n$ which converges in the weak$^*$ topology.
\end{lemma}
\begin{proof} We follow closely the idea of the proof of \cite[Proposition 4.11]{iv}. By countability of the set of cylinders, we can find a subsequence $(\mu_{n_\ell})_\ell$ such that $\lim_{\ell\to\infty}\mu_{n_\ell}(C)=:F(C)$ is well defined for every cylinder $C\subseteq \Sigma$. We  claim that $F$ extends to a countably additive measure on $\Sigma$. As explained in \cite[Proposition 4.11]{iv}, it is enough to prove that 
\begin{align}\label{valericca}\lim_{k\to\infty}\lim_{\ell\to\infty}\mu_{n_\ell}\big(\bigcup_{s\ge k}[a_1,\ldots, a_m,s]\big)=0,\end{align}
for every cylinder $[a_1,\ldots, a_m]$. Since $\mu_n([s])=\mu_n(\sigma^{-m}[s])\ge \mu_n([a_1,\ldots, a_m,s])$, it follows that (\ref{pivale}) implies that (\ref{valericca}). We conclude that $F$ extends to a countably additive measure $\mu$, and that $(\mu_{n_\ell})_\ell$ converges on cylinders to $\mu\in \M_{\le1}(\sigma)$ (see \cite[Proposition 4.11 and Lemma 4.1]{iv}). Observe that (\ref{pivale}) implies $\lim_{k\to\infty}(1-\mu(\bigcup_{s<k} [s]))=0$, and therefore $\mu$ is a probability measure. Since $(\mu_{n_\ell})_\ell$ converges on cylinders to a probability measure we conclude that the sequence converges in the weak$^*$ topology. \end{proof}

We apply Lemma \ref{lem:nicev} to prove that, for the full shift, the escape of mass can be completely ruled out in certain cases.  

\begin{proposition}\label{lem:finalfull} Let $(\N^\N, \sigma)$ be the full shift and $\psi\in C_{uc}(\Sigma)$ a potential such that $\text{var}_1(\psi)$ and $P(\psi)$ are finite. Let $(\mu_n)_n$ be a sequence in $\M(\sigma)$ such that $$P(\psi)=\lim_{n\to\infty}\big(h_{\mu_n}(\sigma)+\int \psi d\mu_n\big).$$ Then, $(\mu_n)_n$ has a subsequence which converges in the weak$^*$ topology. 
\end{proposition}

\begin{proof} Maybe after passing to a subsequence, we can assume that $\lim_{n\to\infty}\mu_n(C)$ is well defined for every cylinder $C\subseteq\Sigma$. By Lemma \ref{lem:nicev}, it is enough to prove that $$\lim_{k\to\infty}\lim_{n\to\infty}\mu_n\big(\bigcup_{s\ge k}[s]\big)=0.$$
We argue by contradiction and suppose this is not the case.  Since the sequence \\$(\lim_{n\to\infty}\mu_n\big(\bigcup_{s\ge k}[s]\big))_k$ is decreasing, there exists $\e>0$ such that $$\lim_{n\to\infty}\mu_n\big(\bigcup_{s\ge k}[s]\big)\ge \e,$$ for every $k\in \N$. Set $A_k=\bigcup_{s\ge k}[s]$. In particular,
$$P\big(\psi+t1_{A_k}\big)\ge \limsup_{n\to\infty} \big(h_{\mu_n}(\sigma)+\int \psi d\mu_n+t\mu_n(A_k)\big)\ge P(\psi)+t\e,$$
for all $t\in [0,+\infty)$ and $k\in\N$. Consider $\psi_0(x)=\sup\{\psi(y):y\in [x_1]\}$, where $x=(x_1,\ldots)\in\N^\N$. Note that  since  $\text{var}_1(\psi)$ is finite, $\|\psi-\psi_0\|_0=: M<\infty$. Also note that $\text{var}_1(\psi_0)=0$. The inequality above implies that 
$$P\big(\psi_0+t1_{A_k}\big)\ge P(\psi_0)+t\e-2M,$$
for all $t\in [0,+\infty)$ and $k\in\N$. Note that $P(\psi_0)\le P(\psi)+M<\infty$. By Example \ref{ex:full}, we have that $P(\psi_0)=\log(\sum_{i\in\N}e^{\psi_0(i)})$, and therefore the series $\sum_{i\in\N}e^{\psi_0(i)}$ converges. Set $t_0\in [0,+\infty)$ such that $(t_0\e-2M)>2$, and $m\in \N$ such that $\sum_{i\ge m}e^{\psi_0(i)}<e^{P(\psi_0)-t_0}$. It follows that 
\begin{align*} 
P(\psi_0)+2< P(\psi_0)+(t_0\e-2M)&<P(\psi_0+t_01_{A_m})\\
&=\log\big(\sum_{i<m}e^{\psi_0(i)}+\sum_{i\ge m}e^{\psi_0(i)+t_0}\big)\\
&\le \log\big(\sum_{i<m}e^{\psi_0(i)}+e^{P(\psi_0)}\big)\\
&\le \log(2e^{P(\psi_0)})\\
&=P(\psi_0)+\log2,
\end{align*}
which is a contradiction. We conclude that the sequence $(\mu_n)_n$ satisfies (\ref{pivale}) and by Lemma \ref{lem:nicev} we conclude that $(\mu_n)_n$ has a subsequence which converges in the weak$^*$ topology to an invariant probability measure. 
\end{proof}

\end{document}